\documentclass[a4paper, UKenglish, numberwithinsect, thm-restate]{tgdk-v2021}

\pdfoutput=1 

\title{High order positivity-preserving numerical methods for a non-local photochemical model} 

\author{Mario Pezzella}{National Research Council of Italy, Institute for Applied Mathematics ``Mauro Picone'', Via P. Castellino, 111 - 80131 Naples, Italy \and Member of the Italian National Group for Scientific Computing GNCS of the National Institute for Advanced Mathematics INdAM}{mario.pezzella@cnr.it}{https://orcid.org/0000-0002-1869-945X}{}

\authorrunning{M. Pezzella} 

\Copyright{Mario Pezzella} 



\category{} 

\relatedversion{} 

\acknowledgements{This work has been performed under the Project PE 0000020 CHANGES - CUP\_B53C22003780006, NRP Mission 4 Component 2 Investment 1.3, Funded by the European Union - NextGenerationEU and in the auspices of the Italian National Group for Scientific Computing (GNCS) of the National Institute for Advanced Mathematics (INdAM).}

\Volume{0}
\Issue{0}
\Article{Pag}
\DateSubmission{January 8, 2025}
\DateAcceptance{}
\DatePublished{}
\SectionAreaEditor{}

\usepackage{bm}
\DeclareMathOperator\erf{erf}
\usepackage{chemfig}

\usepackage{mathtools}

\begin{document}

\maketitle

\begin{abstract}
In this paper we design high-order positivity-preserving approximation schemes for an integro-differential model describing photochemical reactions. Specifically, we introduce and analyze three classes of dynamically consistent methods, encompassing non-standard finite difference schemes, direct quadrature techniques and predictor-corrector approaches. The proposed discretizations guarantee the positivity, monotonicity and boundedness of the solution regardless of the temporal, spatial and frequency stepsizes. Comprehensive numerical experiments confirm the theoretical findings and demonstrate the efficacy of the proposed methods in simulating realistic photochemical phenomena.
\end{abstract}

\section{Introduction}\label{sec:Intro}

Photochemistry, whose origin traces back to the pioneering works of Grotthus and Draper in the $\text{19}^{\text{th}}$ century \cite{Persico_2018}, investigates reactions and physicochemical phenomena induced by light absorption. This field includes natural processes as well as applications across various domains, such as industry, technology and pharmaceuticals (we refer to \cite{rohatgi1978fundamentals,tonnesen1996photostability} for a comprehensive overview). Several mathematical frameworks, primarily focusing on differential systems, have been proposed to model complex light-matter interactions (as detailed in \cite{Baranov2014,Photochemical_MAth_1,Photochemical_MAth_4,Ceseri_Natalini_Pezzella,Photochemical_MAth_3,MAAFI20163537,Photochemical_MAth_2} and references therein).

Numerical methods play a key role in providing efficient tools to simulate the evolution of photochemical phenomena, as analytical solutions of the mathematical models are often not explicitly determinable. The nature of the involved systems, however, imposes the additional challenge of preserving the non-negativity of chemical concentrations.  Techniques such as clipping \cite{SANDU2001589}, which involve setting negative concentrations to zero, are generally unsuitable due to the introduction of artificial mass as a numerical artifact. Moreover, achieving positivity through standard approximation schemes may require impractically small time steps and lead to high computational demands. Therefore, the necessity arises of unconditionally positive numerical integrators that yield positive solutions independently of the chosen discretization steplength \cite{Blanes,Formaggia_Scotti,SCALONE2022351,Zhang2024-br}.  

The preservation of the inherent properties of differential systems is the founding feature of geometric numerical integration \cite{Hairer_2006,GeCO_Izgin,GeC01} and of Non-Standard Finite Difference (NSFD) discretizations \cite{Lubuma_Chem,Conte2022,Mickens_Book,Patidar02062016}. Although the dynamic consistency of such schemes is guaranteed, they suffer from the limitation of low convergence order \cite{Bolley1978,Hoang03102023,NSFD_High}. In general, the definition of high-order positivity-preserving methods is challenging and often restricted to specific types of systems (see, for instance, \cite{PDS,MPV_Axioms,MPV_mixing,OFFNER202015}).

In this work, we focus on the dynamically consistent simulation of an integro-differential model describing the evolution of a general photochemical reaction. Specifically, we address the light-induced conversion of a primary reactant, \texttt{A}, into a product compound, \texttt{B}. Non-linear integral operators are here employed to capture the non-local nature of radiation, accounting for its interaction with the medium in a one-dimensional framework. This paper aims to develop high-order, unconditionally positive numerical methods for the proposed model. To this end, we draw up and compare various schemes based on NSFD approximations, Predictor-Corrector strategies (PC) and Direct Quadrature (DQ) methods.

The manuscript is structured as follows. In Section \ref{sec:Continuous_Model} we report the main results on the integro-differential photochemical model that extends the findings in \cite{Ceseri_Natalini_Pezzella}. In Section \ref{sec:NSFD}, we formulate a NSFD numerical method and provide theoretical results on its consistency and convergence. Furthermore, we prove the preservation of positivity, monotonicity and boundedness for any values of the spatial, temporal and frequency stepsizes. In Section \ref{sec:DQ}, we reformulate the model as a non-linear implicit Volterra integral equation \cite{MPV_implicit} and discretize it using a DQ approach. For the resulting approximation scheme, we prove high order convergence, as well as the existence of a unique, unconditionally positive and bounded numerical solution. In Section \ref{sec:PC} a less demanding, dynamically consistent and quadratically convergent PC method is presented. Numerical experiments are reported in Section \ref{sec: Numerical_Simulations}, where the theoretical properties of the proposed integrators are verified and a comparative analysis of their performance is conducted. Additionally, we present realistic simulations based on experimental outcomes from the literature, pertaining to two phenomena: the photoactivation of serotonin and the photodegradation of cadmium pigments. Final remarks and insights on future development, in Section \ref{sec:Conclusions}, conclude the paper.


\section{The continuous model}\label{sec:Continuous_Model}
Let $c_\texttt{A}(x,t)$ and $c_\texttt{B}(x,t)$ denote the concentrations of the reactant \texttt{A} and the product \texttt{B}, respectively, at spatial position $x \in [0,L]$ and time $t \in [0,T].$ Let $e(x)$ represent the environmental conditions required to initiate the reaction, which could include the presence of activating agents. The evolution of the concentrations, assuming a first-order kinetic model, is then described  by the following conservative Production-Destruction System (PDS)
\begin{equation}\label{eq:PDS_Chem}
    \begin{cases}
        \dfrac{\partial c_\texttt{A}}{\partial t}(x,t)=-K\!\left(c_\texttt{A}(x,t),c_\texttt{B}(x,t),e(x)\right) \ c_\texttt{A}(x,t), \\[0.20cm]
        \dfrac{\partial c_\texttt{B}}{\partial t}(x,t)=K\!\left(c_\texttt{A}(x,t),c_\texttt{B}(x,t),e(x)\right) \ c_\texttt{A}(x,t), \qquad\qquad (x,t)\in [0,L]\times [0,T],
    \end{cases}
\end{equation}
with $c_\texttt{A}(x,0)$ and $c_{\texttt{B}}(x,0)$ given initial conditions. Here, $K(\cdot)$ is the non-local, concentrations-dependent photocatalytic reaction rate,  which will be defined in detail later. Since the reaction involves only two principal chemical species, Lavoisier's conservation principle leads to the linear invariant 
\begin{equation}\label{eq:Linear_Invariant}
    c_\texttt{A}(x,t)+c_\texttt{B}(x,t)=c_\texttt{A}(x,0)+c_\texttt{B}(x,0), \qquad\qquad \text{for all} \ (x,t)\in [0,L]\times [0,T],
\end{equation}
which is a direct consequence of the PDS formulation \eqref{eq:PDS_Chem}, as it follows from $\partial_t\!\left(c_\texttt{A}(x,t)+c_\texttt{B}(x,t)\right)=0$ (we refer to \cite{anderHeiden1982,PDS} and references therein for a comprehensive description of PDS). Furthermore, under the assumption that no product is present at the onset of the reaction (i.e. for $t=0$), we have from \eqref{eq:Linear_Invariant},
\begin{equation}\label{eq: c_B}
     c_\texttt{B}(x,t)=c_\texttt{A}(x,0)-c_\texttt{A}(x,t), \qquad \qquad \qquad \qquad \  \text{for all} \ (x,t)\in [0,L]\times [0,T],
\end{equation}
which allows us to disregard the equation in \eqref{eq:PDS_Chem} for the product \texttt{B}.

The photocatalytic reaction rate $K(\cdot)$ incorporates the dynamic effects of several factors. First, the dependence of the process on the system temperature and environmental conditions is expressed, following Arrhenius' law (see, for instance, \cite[Section 10.9]{Atkins-DePaula} or \cite[Page 174]{Laidler}), through the term
\begin{equation*}\label{eq:Arrhenius}
f(x) = \mathcal{A} \exp\left\{-\frac{E_a}{R \, \mathcal{T}_K}\right\} e(x),
\end{equation*}
where $\mathcal{A}$ is the Arrhenius pre-exponential factor, $E_a$ is the activation energy of the reaction, $R$ is the gas constant and $\mathcal{T}_K$ is the temperature. We further consider that the chemical process is exclusively triggered by light within a specific wavelength range $[\lambda_{0},\lambda_{*}],$ defined by the photophysical properties of the system, which corresponds to photon energies sufficiently high to trigger the reaction. Let $\varepsilon_\texttt{A}$ and $\varepsilon_\texttt{B}$ denote the molar absorption coefficients for the reagent $\texttt{A}$ and the product $\texttt{B}$, respectively. These functions measure how strongly each chemical species absorbs, and therefore attenuates, light at a given wavelength $\lambda.$ The total absorbance of the system is then given by
\begin{equation*}
     \mu_s(\lambda,c_\texttt{A}(z,t),c_\texttt{B}(z,t))=\varepsilon_\texttt{A}(\lambda)c_\texttt{A}(z,t)+\varepsilon_\texttt{B}(\lambda)c_\texttt{B}(z,t)=(\varepsilon_\texttt{A}(\lambda)-\varepsilon_\texttt{B}(\lambda))c_\texttt{A}(z,t)+\varepsilon_\texttt{B}(\lambda)c_\texttt{A}(x,0),
\end{equation*}
where $\lambda\in[\lambda_{0},\lambda_{*}]$ represents the wavelength of the radiation. Here, in accordance with the Beer-Lambert law \cite{d2012simplified}, the overall light intensity is modeled as  
\begin{equation*}
    \iota  \left(  \lambda,\, C_0(x),  \int_0^x   c(\xi,t) \ d\xi  \right)  =  I(\lambda) \exp  \left\{  -  \mu  \left(  \varepsilon_\texttt{B}(\lambda) C_0(x) +  (\varepsilon_\texttt{A}(\lambda) -\varepsilon_\texttt{B}(\lambda))  \int_0^x    c(\xi,t)  \ d\xi  \right)  \right\},
\end{equation*}
where $C_0(x) = \int_0^x  c_\texttt{A}(\xi,0) \ d\xi$ is given, $I(\lambda) = I_0(\lambda)/I_{ref}$ denotes the normalized irradiance and $\mu > 0$ is a dimensionless attenuation constant. Furthermore, inspired by the approaches in \cite{Ceseri_Natalini_Pezzella,Clarelli2013,EILERS,THEBAULT}, we introduce the function  
\begin{equation}\label{eq: rho}
    \rho(X) = \dfrac{a_1 X}{X^2 + a_2 X + 1}, \qquad \text{with} \qquad a_1>0, \qquad  a_2 \geq 0, 
\end{equation}
and consequently model the overall light penetration effect with the non-local term  
\begin{equation*}
    \scalebox{1.3}{$\displaystyle\int$}_{\!\!\!\!\! \lambda_{0}}^{\lambda_{*}} \!\!\! \rho\left(\iota\left(\lambda,\, C_0(x), \int_0^x    c(\xi,t)  \ d\xi\right)\right) \ d\lambda.
\end{equation*}

The aggregation of the contributions outlined above into the PDS \eqref{eq:PDS_Chem} yields the following integro-differential model
\begin{equation}\label{eq:continuous_model}
    \begin{cases}
         \displaystyle\dfrac{\partial c}{\partial t}(x,t)=- c(x,t) f(x)\scalebox{1.3}{$\displaystyle\int$}_{\!\!\!\!\! \lambda_{0}}^{\lambda_{*}} \!\!\! \rho\left(\iota\left(\lambda,\, C_0(x), \int_0^x    c(\xi,t)  \ d\xi\right)\right) \ d\lambda, \qquad\qquad \ \ (x,t)\in [0,L]\times [0,T], \\
         \displaystyle \iota  \left(  \lambda,\, C_0(x),  \int_0^x   c(\xi,t) \ d\xi  \right)  =  I(\lambda) \exp  \left\{  -  \mu  \left(  \varepsilon_\texttt{B}(\lambda) C_0(x) +  (\varepsilon_\texttt{A}(\lambda) -\varepsilon_\texttt{B}(\lambda))  \int_0^x    c(\xi,t)  \ d\xi  \right)  \right\},
    \end{cases}
\end{equation}
where, for the sake of brevity, the notation $c(x,t)$ is employed in place of $c_{\texttt{A}}(x,t).$ Accordingly, once \eqref{eq:continuous_model} is solved, the product concentration $c_{\texttt{B}}(x,t)$ is retrieved from the conservation law \eqref{eq: c_B}.

Hereafter, our theoretical investigation will be conducted under the following assumptions for the known functions of model \eqref{eq:continuous_model}.

\begin{description}
\item[Assumptions (A)]\phantomsection\label{ass:1} $f\in C^0([0,L]\to \mathbb{R}^+_0),$ $I\in C^0([\lambda_{0},\lambda_{*}]\to [0,1]),$  $\rho\in C^0(\mathbb{R}^+_0 \to \mathbb{R}^+_0)$ are non-negative functions. Furthermore, the given initial state $c^0(x)=c_{\texttt{A}}(x,0)\in C^0([0,L]\to \mathbb{R}^+)$ is a  positive function.
\end{description}

The existence and uniqueness of a solution to \eqref{eq:continuous_model} is established, assuming a Lipschitz-continuity property for the known functions, via an extension of the Picard-Lindel\"{o}f theorem to Banach spaces (see, for instance, \cite[Theorem 6]{ODE_Exs} or \cite[Satz 1.17]{Tedesco_Libro} and references therein). 
As a matter of fact, denoted by $\mathsf{c}_x(t) : t\in \mathbb{R}^+_0 \to c(x,t)\in L^1([0,L]),$ the model \eqref{eq:continuous_model} corresponds to a functional ordinary differential equation of the form $d \mathsf{c}_x/dt=F(\mathsf{c}_x), $ with  $F: L^1([0,L])\to L^1([0,L])$ Lipschitz-continuous operator derived from the right-hand side of \eqref{eq:continuous_model}.  Furthermore, the positivity condition for the initial state in (\hyperref[ass:1]{A}) excludes the possibility of a stationary solution. Standard contradiction arguments yield the following positivity and monotonicity properties
\begin{equation}\label{eq:Properties}
   0<c(x,t)<c(x,\tau )< c^0(x), \quad \text{for all} \quad 0\leq x \leq L \quad \text{and} \quad 0\leq \tau <t \leq T,   
\end{equation}
whose well-established physical meaning motivates our interest in designing dynamically consistent numerical strategies.

\section{A non-standard finite difference scheme}\label{sec:NSFD}
Let $\Delta x, \Delta t, \Delta \lambda\in \mathbb{R}^+,$ denote the spatial, temporal and frequency stepsizes, respectively. Consider the uniform meshes $\{x_j=j \Delta x, \ j=0,1,\dots \},$ $\{t_n=n \Delta t, \ n=0,1,\dots \}$ and $\{\lambda_l=\lambda_{0}+l \Delta \lambda, \ l=0,1,\dots \}.$ We define the following discretization scheme for \eqref{eq:continuous_model}
\begin{equation}\label{eq:NSFD_scheme}
    \begin{cases}
        \displaystyle\dfrac{c_j^{n+1}-c_j^n}{\phi(\Delta t)}=- c_j^{n+1} f(x_j) \ \Delta \lambda \! \sum_{l=0} ^{N_\lambda-1} \! \rho\left(\iota\left(\lambda_l,\, C_0(x_j) ,\, \Delta x \sum_{r=0}^{j-1}c_r^n \right)\right), 
        \\
        \displaystyle \! \iota \! \left(\lambda_l, C_0(x_j) , \Delta x \sum_{r=0}^{j-1}c_r^n \!\right) \! =I(\lambda_l) \exp \! \left\{ \! - \mu \! \left(  \varepsilon_\texttt{B}(\lambda_l) C_0(x_j) +  (\varepsilon_\texttt{A}(\lambda_l) -\varepsilon_\texttt{B}(\lambda_l))  \Delta x \sum_{r=0}^{j-1}c_r^n  \right) \! \right\} ,
    \end{cases}
\end{equation}
where the positive initial values $c_j^0=c^0(x_j)=c(x_j,0)$ are given, $c_j^n\approx c(x_j,t_n)$ for $n,j\geq 0$ and
\begin{equation}\label{eq:prop_phi}
    \phi \; : \; \Delta t\in \mathbb{R}^+ \longrightarrow \; \phi(\Delta t)=\Delta t + \mathcal{O}(\Delta t^2)\in \mathbb{R}^+.
\end{equation}
The numerical method \eqref{eq:NSFD_scheme} is derived from a non-standard approximation of the first derivative operator, combined with a non-local implicit-explicit discretization of the right-hand side of \eqref{eq:continuous_model}. Consequently, in accordance with the definitions provided in \cite{Anguelov, SUNDAY}, it falls within the class of Non-Standard Finite Difference (NSFD) discretizations, which were originally introduced for differential equations (see \cite{Mickens_PDE,Mickens_Book_Old,Mickens_Book} and references therein) and only recently extended to integral and integro-differential problems \cite{Buonomo_2024,Lubuma2015, MPV_JCD, MPV_MBE, Pandey_2015}.

The following outcome addresses the unconditional positivity, boundedness and monotonicity of the NSFD solution.
\begin{theorem}\label{thm:NSFD_Dynamic_Consistency}
    Consider the discrete equations \eqref{eq:NSFD_scheme} under the assumptions (\hyperref[ass:1]{A}) and \eqref{eq:prop_phi}. Then, independently of $\Delta x, \Delta t, \Delta \lambda\in \mathbb{R}^+$ and for each $j\in \mathbb{N}_0,$ the  sequence $\{ c_j^n \}_{n\in \mathbf{N}_0}$ is positive, bounded from above and monotonically decreasing.
\end{theorem}
\begin{proof}
    We proceed by mathematical induction to prove that, for any integer value of $j,$ the statement $0<c_j^{n+1}<c_j^n\leq  c_j^0$ holds $\forall n\in \mathbb{N}_0.$ Since the initial concentration is positive, the base case $n=0$ naturally follows from the assumptions (\hyperref[ass:1]{A}). Consider $n>0$ and assume that the properties are verified for each $0\leq m \leq n-1.$ The result then comes from 
    \begin{equation}\label{eq:NSFD_scheme_rapporto}
        c_j^{n+1}=\dfrac{c_j^n}{ 1+ \phi(\Delta t) \ \Delta \lambda \ f(x_j) \sum_{l=0}^{N_\lambda-1}\rho\left(\iota\left(\lambda_l,\, C_0(x_j) ,\, \Delta x \sum_{r=0}^{j-1}c_r^n \right)\right)},
    \end{equation}
    where, for the induction hypotheses, $0<c_j^{n}\leq c_j^{0}$ and $\rho\left(\iota\left(\lambda_l,\, C_0(x_j) ,\, \Delta x \sum_{r=0}^{j-1}c_r^n \right)\right)>0.$ 
\end{proof}
As pointed out in the previous section, under the assumptions (\hyperref[ass:1]{A}), the continuous solution to \eqref{eq:continuous_model} is positive and decreasing with respect to time. Hence, Theorem \ref{thm:NSFD_Dynamic_Consistency} demonstrates the dynamical consistency \cite{Mickens_Dynamic_Consistency} of the NSFD scheme \eqref{eq:NSFD_scheme} ensuring that, whatever the stepsizes and the denominator function in \eqref{eq:prop_phi}, a discrete counterpart of the properties \eqref{eq:Properties} holds for the numerical solution. Furthermore, due to the linearly implicit structure of \eqref{eq:NSFD_scheme}, its reformulation as \eqref{eq:NSFD_scheme_rapporto} leads to a straightforward and computationally efficient algorithm.

The non-standard nature of the discretization \eqref{eq:NSFD_scheme} makes the investigation of the local truncation error
\begin{equation}\label{eq:local_error}
    \begin{split}
        \delta_j^n&=\delta(\Delta x, \Delta t, \Delta \lambda; \,  x_j,t_n) \\ 
        &=\displaystyle\dfrac{\partial c}{\partial t}(x_j,t_n)+ c(x_j,t_n) f(x_j)\scalebox{1.3}{$\displaystyle\int$}_{\!\!\!\!\! \lambda_{0}}^{\lambda_{*}} \!\!\! \rho\left(\! \iota \!\left(\lambda,\, C_0(x_j), \int_0^{x_j}    c(\xi,t_n)  \ d\xi\right)\!\right) \ d\lambda \\
        &- \dfrac{c(x_j,t_{n}+\Delta t)-c(x_j,t_{n})}{\phi(\Delta t)} - c(x_j,t_{n}+\Delta t) f(x_j) \ \Delta \lambda \!\! \sum_{l=0} ^{N_\lambda-1} \!\! \rho\left( \! \iota \! \left(\lambda_l,\, C_0(x_j) ,\, \Delta x \sum_{r=0}^{j-1}c(x_r,t_{n}) \right)\!\right) \! ,
    \end{split}
\end{equation}
more challenging. Its behavior is analyzed through the following consistency result.

\begin{lemma}\label{lemma:NSFD_Cons}
    Assume that the known functions describing problem \eqref{eq:continuous_model} are continuously differentiable for $x\in [0,L],$ $t \in [0,T]$ and $\lambda \in [\lambda_{0},\lambda_{*}].$ Then the NSFD method \eqref{eq:NSFD_scheme} is consistent with \eqref{eq:continuous_model}, of order 1. 
\end{lemma}

\begin{proof}
    Let $N_x,$ $N_t$ and $N_\lambda$ be positive integers such that $L=N_x \Delta x,$ $T=N_t \Delta t$ and $\lambda_{*} =\lambda_{0}+ N_\lambda \Delta \lambda.$
    The regularity assumptions on the know functions imply that $c(x_j,t)\in C^2([0,T])$ and therefore, from the mean value theorem, there exist $\theta_j^n\in (0,1)$ and $\eta_j^n\in (0,1)$ such that 
    \begin{equation*}
        c(x_j,t_{n}+\Delta t)-c(x_j,t_{n})=\Delta t \ \dfrac{\partial c}{\partial t}(x_j,t_{n}+\theta_j^n\Delta t)=\Delta t \ \dfrac{\partial c}{\partial t}(x_j,t_{n}) + \Delta t^2 \theta_j^n \ \dfrac{\partial^2 c}{\partial t^2}(x_j,t_{n}+\theta_j^n\eta_j^n\Delta t), 
    \end{equation*} 
    for each $n=0,\dots, N_t-1$ and $j=0,\dots, N_x.$
    Furthermore, since $c(x,t_n)\in C^1([0,L])$, the convergence of the rectangular quadrature rule (see, for instance, \cite{Davis}) ensures the existence of positive constants $\kappa_1$ and $\kappa_2$ such that for the terms
    \begin{equation*}
        \begin{split}
            \varrho(\Delta \lambda; x_j, t_n)&= \!\! \scalebox{1.3}{$\displaystyle\int$}_{\!\!\!\!\! \lambda_{0}}^{\lambda_{*}} \!\!\!\!\!\! \rho \! \left(\! \iota \! \left(\lambda, C_0(x_j), \int_0^{x_j} \!\! c(\xi,t_n)  \ d\xi\right)\!\right) \, d\lambda - \Delta \lambda \! \!\sum_{l=0} ^{N_\lambda-1} \! \! \rho \! \left( \! \iota \! \left(\lambda_l, C_0(x_j) ,\, \int_0^{x_j}  \!\!  c(\xi,t_n)  \ d\xi \right)\!\right) \! ,  \\
            \chi(\Delta x; x_j, t_n)&=\int_0^{x_j}    c(\xi,t_n)  \ d\xi-\Delta x \sum_{r=0}^{j-1}c(x_r,t_{n}),
        \end{split}
    \end{equation*}
    the inequalities $|\varrho(\Delta \lambda; \, x_j, t_n)|\leq \kappa_1 \Delta \lambda$ and $|\chi(\Delta x; \, x_j, t_n)|\leq \kappa_2 \Delta x$ hold true, independently of $n=0,\dots, N_t$ and $j=0,\dots, N_x.$ Hence, recalling the properties \eqref{eq:Properties}, straightforward manipulations of the local truncation errror \eqref{eq:local_error} lead to
    \begin{equation*}
        \begin{split}
            |\delta(\Delta x, \Delta t, \Delta \lambda; \,  x_j,t_n)|\leq & -\left|1-\dfrac{\Delta t}{\phi(\Delta t)}\right| \, \dfrac{\partial c}{\partial t}(x_j,t_n)+\dfrac{\Delta t^2}{\phi(\Delta t)} \, \left|\dfrac{\partial^2 c}{\partial t^2}(x_j,t_{n}+\theta_j^n\eta_j^n\Delta t)\right| \\
            & +\kappa_3 \left( \left| \varrho(\Delta \lambda; x_j, t_n) \right| +  \lambda_{*}  \left( \left|\chi(\Delta x; x_j, t_n) \right| - \Delta t \dfrac{\partial c}{\partial t}(x_j,t_{n}+\theta_j^n\Delta t) \right)  \right),
        \end{split}
    \end{equation*}
with $\kappa_3>0$ depending on the bounds of $\rho$, $\iota$ and the known functions in \eqref{eq:continuous_model}, as well as on their derivatives, but not on the stepsizes. Finally, since  \eqref{eq:prop_phi} implies $\frac{\Delta t}{\phi(\Delta t)}=1+\mathcal{O}(h),$ we conclude
\begin{equation*}
        \max_{\substack{n=0,\dots, N_t \\ j=0,\dots, N_x}}|\delta(\Delta x, \Delta t, \Delta \lambda; \,  x_j,t_n)| \leq \kappa (\Delta t+\Delta x + \Delta \lambda),
\end{equation*}
with $\kappa>0$ depending on $\kappa_1,$ $\kappa_2$ and the bounds on the time derivatives of the solution.
\end{proof}

With the following theorem, we prove the convergence of the NSFD scheme \eqref{eq:NSFD_scheme} and show that the corresponding global discretization error
\begin{equation}\label{eq:global_error}
   \mathsf{e}_j^n=\mathsf{e}(\Delta x, \Delta t, \Delta \lambda; \,  x_j,t_n)=c(x_j,t_n)-c_j^n,    \qquad n=0,\dots, N_t, \quad j=0,\dots, N_x,
\end{equation}
vanishes linearly with the sum of the stepsizes. 
\begin{theorem}\label{thm:NSFD_conv} 
    Assume that the given functions in \eqref{eq:continuous_model} are continuously differentiable for $x\in [0,L],$ $t \in [0,T]$ and $\lambda \in [\lambda_{0},\lambda_{*}].$ Let $\{c^n_j\}$ be the approximations of the continuous solution to \eqref{eq:continuous_model} computed by the NSFD scheme \eqref{eq:NSFD_scheme} with $\Delta x=L/N_x,$ $\Delta t=T/N_t,$ $\Delta \lambda=(\lambda_{*}-\lambda_{0})/N_\lambda$ and $N_x,N_t,N_\lambda\in \mathbb{N}.$ Then
    \begin{equation*}
        \max_{\substack{n=0,\dots, N_t \\ j=0,\dots, N_x}}|\mathsf{e}(\Delta x, \Delta t, \Delta \lambda; \,  x_j,t_n)|\longrightarrow 0, \qquad \text{as} \qquad \Delta x+\Delta t+\Delta \lambda \to 0.
    \end{equation*}
    Furthermore, the NSFD method \eqref{eq:NSFD_scheme} is convergent of order $1.$
\end{theorem}
\begin{proof}
    Evaluating equation \eqref{eq:continuous_model} at $(x_j,t_n),$ for $n=0,\dots, N_t$ and $j=0,\dots, N_x,$ and subtracting \eqref{eq:NSFD_scheme} from it yields
    \begin{equation*}
        \begin{split}
            \mathsf{e}_j^{n+1}\!&=\mathsf{e}_j^{n}-\phi(\Delta t) \left\{\delta_j^n+ \mathsf{e}_j^{n+1} f(x_j) \ \Delta \lambda \! \sum_{l=0} ^{N_\lambda-1} \!\! \rho\left( \! \iota \! \left(\lambda_l,\, C_0(x_j) ,\, \Delta x \sum_{r=0}^{j-1}c(x_r,t_{n}) \right)\!\right) \right\} \\
            & -\phi(\Delta t) c_j^{n+1} f(x_j)   \Delta \lambda \!\! \sum_{l=0} ^{N_\lambda-1} \!\! \left\{ \! \rho \! \left( \! \iota \! \left(\lambda_l, C_0(x_j) , \Delta x \sum_{r=0}^{j-1}c(x_r,t_{n}) \right)\!\right) - \rho \! \left( \! \iota \! \left(\lambda_l, C_0(x_j) , \Delta x \sum_{r=0}^{j-1}c_r^n \right)\!\right) \! \right\} \! ,
        \end{split}
    \end{equation*}
    with $\delta_j^n$ and $\mathsf{e}_j^n$ denoting the local and global discretization errors defined in \eqref{eq:local_error} and \eqref{eq:global_error}, respectively. The regularity of the given functions, the bound on the NSFD numerical solution (see Theorem \ref{thm:NSFD_Dynamic_Consistency}) and the property in \eqref{eq:prop_phi} assure the existence of positive constants $\kappa_1$ and $\kappa_2$ such that, for a sufficiently small temporal stepsize $\Delta t,$ the inequality
    \begin{equation*}
        |\mathsf{e}_j^{n+1}| \leq \dfrac{\phi(\Delta t)}{1-\kappa_1 \phi(\Delta t)} |\delta_j^n| +  \dfrac{\kappa_2 \Delta x}{1-\kappa_1 \phi(\Delta t)}\sum_{r=0}^{j-1}|\mathsf{e}_r^{n}|+\dfrac{|\mathsf{e}_j^{n}|}{1-\kappa_1 \phi(\Delta t)}, \qquad \qquad
        \begin{array}{l}
            n=0,\dots, N_t-1,  \\
            j \, =0,\dots, N_x, 
        \end{array}
    \end{equation*}
    holds true. From the discrete Gronwall-type inequality in \cite[Theorem 11.1]{Linz_1985},
    \begin{equation*}
        \max_{0\leq n \leq N_t}|\mathsf{e}_j^n| \leq  \left( \dfrac{(1+\kappa_2 N_x N_\lambda)^{N_x}-1+\kappa_1 \phi(\Delta t)}{\kappa_1 (1-\kappa_1 \phi(\Delta t))^{N_x}}\right) \max_{0\leq n \leq N_t}|\delta_j^n| , \qquad \quad \quad \text{for each} \;\; j \, =0,\dots, N_x
    \end{equation*}
    and therefore, from the consistency outcome of Lemma \ref{lemma:NSFD_Cons}, 
    \begin{equation*}
            \max_{\substack{n=0,\dots, N_t \\ j=0,\dots, N_x}}|\mathsf{e}(\Delta x, \Delta t, \Delta \lambda; \,  x_j,t_n)| \leq \Tilde{\kappa} (\Delta t+\Delta x + \Delta \lambda),
    \end{equation*}
    which completes the proof.
\end{proof}

\section{Direct quadrature method}\label{sec:DQ}
In Section \ref{sec:NSFD} we introduced a linearly implicit, positivity-preserving NSFD discretization for the integro-differential model \eqref{eq:continuous_model}. Despite its dynamical consistency, the first order convergence of \eqref{eq:NSFD_scheme} may represent a limitation for realistic simulations over extended integration intervals. In order to devise high-order unconditionally positive methods, here we adopt a different approach built upon a suitable reformulation of the evolution operator in \eqref{eq:continuous_model}.

In what follows, we address the following non-linear implicit Volterra  integral equation \cite{MPV_implicit}
\begin{equation}\label{eq:VFIE_model}
    \begin{cases}
         \displaystyle \log\left(\dfrac{c(x,t)}{c^0(x)}\right)=-f(x) \scalebox{1.3}{$\displaystyle\int$}_{\!\!\!\!\! 0}^{t}  \scalebox{1.3}{$\displaystyle\int$}_{\!\!\!\!\! \lambda_{0}}^{\lambda_{*}} \!\!\! \rho\left(\iota\left(\lambda,\, C_0(x), \int_0^x    c(\xi,\tau)  \ d\xi\right)\right) \ d\lambda \ d \tau,  \\
         \displaystyle \iota  \left(  \lambda,\, C_0(x),  \int_0^x   c(\xi,t) \ d\xi  \right)  =  I(\lambda) \exp  \left\{  -  \mu  \left(  \varepsilon_\texttt{B}(\lambda) C_0(x) +  (\varepsilon_\texttt{A}(\lambda) -\varepsilon_\texttt{B}(\lambda))  \int_0^x    c(\xi,t)  \ d\xi  \right)  \right\},
    \end{cases} 
\end{equation}
which is equivalent to \eqref{eq:continuous_model} and is derived from it through integration with respect to time, taking into account the properties of the continuous solution in \eqref{eq:continuous_model}.

Consider the uniform meshes $\{x_j=j \Delta x, \ j=0,1,\dots \},$ $\{t_n=n \Delta t, \ n=0,1,\dots  \}$ and $\{\lambda_l=\lambda_{0}+l \Delta \lambda, \ l=0,1,\dots, N_\lambda \},$ where $\Delta x, \Delta t, \Delta \lambda\in \mathbb{R}^+,$ are the spatial, temporal and frequency stepsizes, respectively. We approximate the integral terms in \eqref{eq:VFIE_model} by the means of the $(n_0-1)$-th Gregory quadrature rule (we refer to \cite{MPV_mixing} and references therein for further details), with $n_0\geq 1.$ Let $\{w_{\nu \mu}\},$ $\mu=0,\dots,n_0-1$ and $\{\omega_\nu\},$ $\nu\geq n_0$ be the starting and convolution weights, respectively. As detailed in \cite[Section 2.6]{Brunner}, the weights are positive and satisfy 
\begin{equation}\label{eq:weights_bounds}
	\underset{\nu\geq 0}{\sup}\ \underset{0\leq \mu < n_0}{\max}w_{\nu\mu}\leq \tilde{w} <+\infty 
	\;\;\;\;\;\;\;\;\; \mbox{and} \;\;\;\;\;\;\;\;\; 
	\underset{\nu\geq 0}{\sup}\;\omega_\nu\leq \tilde{\omega} <+\infty.
\end{equation} 
We then define the $n_0$-steps Direct Quadrature (DQ) \cite[Section 3.2]{Brunner} scheme for \eqref{eq:VFIE_model} as follows 
\begin{equation}\label{eq:DQ_method}
    \begin{cases} 
        \log\!\left(\dfrac{c_j^n}{c_j^0}\right)\!=\!\displaystyle -\Delta t \Delta \lambda \, f(x_j) \! \sum_{p=n_0}^{n} \! \omega_{n-p} \Bigg\{\!\sum_{q=0}^{n_0-1} \! w_{N_\lambda q} \ \rho\!\left( \! \iota \! \left(\lambda_q,\, C_0(x_j) ,\, \Delta x \!\! \sum_{k=0}^{n_0-1} \! w_{jk}c_k^p+\Delta x \!\! \sum_{h=n_0}^{j}\omega_{j-h}c_h^p\right)\!\right) \\
        \phantom{log\!\left(\dfrac{c_j^n}{c_j^0}\right)\!=\!} \displaystyle + \sum_{s=n_0}^{N_\lambda}\omega_{N_\lambda-s} \ \rho\!\left( \! \iota \! \left(\lambda_s,\, C_0(x_j) ,\, \Delta x \! \sum_{k=0}^{n_0-1}w_{jk}c_k^p+\Delta x \! \sum_{h=n_0}^{j}\omega_{j-h}c_h^p\right)\!\right) \Bigg\}+ \sigma_j , \\
        \displaystyle \iota \! \left(\!\lambda_l,\, C_0(x_j) ,\, \Delta x \!\! \sum_{k=0}^{n_0-1} \! w_{jk}c_k^p+\Delta x \!\! \sum_{h=n_0}^{j} \! \omega_{j-h}c_h^p\right)\!=I(\lambda_l) \exp \! \Bigg\{ \!\! - \mu  \Bigg( \! \varepsilon_\texttt{B}(\lambda_l) C_0(x_j) +  (\varepsilon_\texttt{A}(\lambda_l) -\varepsilon_\texttt{B}(\lambda_l)) \\
        \phantom{log\!\left(\dfrac{c_j^n}{c_j^0}\right)\!=\!} \displaystyle\left(\Delta x \! \sum_{k=0}^{n_0-1}w_{jk}c_k^p+\Delta x \! \sum_{h=n_0}^{j}\omega_{j-h}c_h^p\right)  \Bigg) \! \Bigg\}, \qquad \qquad \qquad\qquad\qquad\qquad \ l=0,\dots, N_\lambda,
    \end{cases}
\end{equation}
where $c_j^0=c^0(x_j)>0$ and $c_j^n\approx c(x_j,t_n),$ for $n,j\geq n_0,$ the starting values $c_k^m,$ $0\leq k,m\leq n_0-1,$  are given and 
\begin{equation*}
    \begin{split}
        \sigma_j=&-\Delta t \Delta \lambda \, f(x_j)\sum_{m=0}^{n_0-1} w_{nm} \Bigg\{\sum_{i=0}^{n_0-1}w_{N_\lambda i} \ \rho\left( \! \iota \! \left(\lambda_i,\, C_0(x_j) ,\, \Delta x \sum_{k=0}^{n_0-1}w_{jk}c_k^m+\Delta x\sum_{h=n_0}^{j}\omega_{j-h}c_h^m\right)\!\right) \\ &+\sum_{l=n_0}^{N_\lambda}\omega_{N_\lambda-l} \ \rho\left( \! \iota \! \left(\lambda_l,\, C_0(x_j) ,\, \Delta x \sum_{k=0}^{n_0-1}w_{jk}c_k^m+\Delta x\sum_{h=n_0}^{j}\omega_{j-h}c_h^m\right)\!\right) \Bigg\}. 
    \end{split}
\end{equation*}

Here, we demonstrate the well-posedness of the DQ method \eqref{eq:DQ_method}
and prove the unconditional positivity and boundedness of the corresponding numerical solution.

\begin{theorem}\label{thm:DQ_Properties}
    Consider the $n_0$-steps discrete equations in \eqref{eq:DQ_method}, under the assumptions specified in (\hyperref[ass:1]{A}). Suppose that the initial values are given and satisfy $0<c_j^m\leq c^0(x_j)$, for $0 \leq m\leq n_0 - 1$ and $j=0,\dots,N_x.$ Then, for each \(\Delta x, \Delta t, \Delta \lambda \in \mathbb{R}^+\) and $n=n_0,\dots,N_t$, the non-linear system \eqref{eq:DQ_method} admits a component-wise positive solution $\bm{c}^n=[c_0^n , \dots , c_{N_x}^n]^\mathsf{T}$ lying in
    \begin{equation}\label{eq:Omega}
        \Omega = \left\{[x_0,\dots,x_{N_x}]^\mathsf{T} : \ \bar{c}(\Delta x, \Delta t, \Delta \lambda) \leq x_i \leq \max_{0 \leq j \leq N_x}\! c_j^0, \ \ \text{for all} \ \ i = 0, \dots, N_x\right\} \subset \mathbb{R}^{N_x + 1},
    \end{equation}
    where the lower bound is given by
    \begin{equation}\label{eq:DQ_bound_inferiore}
        \bar{c}(\Delta x, \Delta t, \Delta \lambda) = \left(\min_{0 \leq j \leq N_x}\!c_j^0\right) \exp\left\{-\Delta t (N_t + 1) \ \Delta \lambda (N_\lambda + 1) \, \mathsf{W}^2 \, \mathsf{R} \, \mathsf{F}_{\Delta x}\right\},
    \end{equation}
    with $\mathsf{F}_{\Delta x} = \max_{[0, N_x \Delta x]} f(x),$ $\mathsf{R} = \max_{\mathbb{R}^+_0} \rho(x)$ and, recalling \eqref{eq:weights_bounds}, $\mathsf{W}=\max\{\tilde{w}, \tilde{\omega}\}.$ Furthermore, such a solution is unique if $\partial_x f(x_j)\geq 0$ and $\varepsilon_{\texttt{B}}(\lambda_l)\geq\varepsilon_{\texttt{A}}(\lambda_l)$ hold true for each $j=0,\dots, N_x$ and $l=0,\dots, N_\lambda.$
\end{theorem}
\begin{proof}
     We preliminary observe that, from the hypotheses, $\bm{c}^m=\begin{bmatrix}c_0^m & \dots & c_{N_x}^m\end{bmatrix}^\mathsf{T}$ belongs to $\Omega$ for each $m=0,\dots,n_0-1.$ Furthermore, the non-negativity of $a_1$ and $a_2$ in \eqref{eq: rho} and the assumptions (\hyperref[ass:1]{A}) imply that $\mathsf{R}$ and  $\mathsf{F}$ are positive and finite. In what follows, we denote by 
    \begin{equation*}
        \begin{split}
            \bm{\Gamma}_{j}&=[
            w_{j  0} ,\,  w_{j  1} ,\, \dots ,\, w_{j n_0-1} ,\, \omega_{j-n_0} ,\, \omega_{j-n_0-1} ,\, \dots ,\, \omega_0 ,\, 0 ,\, \dots ,\,  0 
            ]^\mathsf{T} \in \mathbb{R}^{N_x+1}, \qquad \  j=n_0,\dots,N_x, \\
            \bm{\Lambda}&=[
            w_{N_\lambda  0} ,\,  w_{N_\lambda  1} ,\, \dots ,\, w_{N_\lambda  n_0-1} ,\, \omega_{N_\lambda-n_0} ,\, \omega_{N_\lambda-n_0-1} ,\, \dots ,\, \omega_0  
            ]^\mathsf{T} \in \mathbb{R}^{N_\lambda+1}, \\
            \bm{\Psi}^n&=[
            w_{n  0} ,\,  w_{n 1} ,\, \dots ,\, w_{n  n_0-1} ,\, \omega_{n-n_0} ,\, \omega_{n-n_0-1} ,\, \dots ,\, \omega_0
            ]^\mathsf{T} \in \mathbb{R}^{n+1}, \qquad \qquad \qquad \ \ n=n_0,\dots,N_t.
        \end{split}
    \end{equation*}
    Consider a fixed time step $n\geq n_0$ and assume that $\bm{c}^\nu \, =[c_0^\nu , c_1^\nu , \dots , c_{N_x}^\nu]^\mathsf{T}\in \Omega,$ for each $n_0\leq \nu \leq n-1.$ Define the functions 
    $\bm{\mathcal{M}}_{j}^n \ : \mathbb{R}^{N_x + 1} \to \mathbb{R}^{(N_\lambda+1)\times(n+1)}$ and $\bm{G}^n \ : \mathbb{R}^{N_x + 1} \to \mathbb{R}^{N_x + 1}$ as follows
    \begin{equation}\label{eq:Forma_G}
        \begin{split}
             \bm{\mathcal{M}}_{j}^n(\bm{x})&=
             \begin{bmatrix}
                 \rho \! \left(  \iota \! \left(\lambda_0, C_0(x_j) , \Delta x  \bm{\Gamma}_{j}^\mathsf{T}\bm{c}^0 \right) \! \right) & \dots &  \rho \! \left(  \iota \! \left(\lambda_{0}, C_0(x_j) , \Delta x  \bm{\Gamma}_{j}^\mathsf{T}\bm{x} \right) \! \right)  \\
                 \rho \! \left(  \iota \! \left(\lambda_1, C_0(x_j) , \Delta x  \bm{\Gamma}_{j}^\mathsf{T}\bm{c}^0 \right) \! \right) & \dots &  \rho \! \left(  \iota \! \left(\lambda_{1}, C_0(x_j) , \Delta x  \bm{\Gamma}_{j}^\mathsf{T}\bm{x} \right) \! \right)\\
                 \dots & \dots  & \dots  \\
                 \rho \! \left(  \iota \! \left(\lambda_{N_\lambda}, C_0(x_j) , \Delta x  \bm{\Gamma}_{j}^\mathsf{T}\bm{c}^0 \right) \! \right) & \dots &  \rho \! \left(  \iota \! \left(\lambda_{N_{\lambda}}, C_0(x_j) , \Delta x  \bm{\Gamma}_{j}^\mathsf{T}\bm{x} \right) \! \right)
            \end{bmatrix}, \qquad \quad n_0\leq j\leq N_x, \\
            \bm{G}^n(\bm{x})&=\bm{c^0} \odot \left( \sum_{k=0}^{n_0-1}\bm{\epsilon}_{k+1}c_k^n+ \sum_{j=n_0}^{N_x} \bm{\epsilon}_{j} \exp\left\{-\Delta t \Delta \lambda \, f(x_j) \, \left(\bm{\mathcal{M}}_{j}^n\!\left(\bm{x}\right) \cdot \bm{\Psi}^n\right)^\mathsf{T} \!\! \cdot \bm{\Lambda}  \right\}\right),
        \end{split}
    \end{equation}
    where $\odot$ denotes the element-wise Hadamard product and $\bm{\epsilon}_j\in \mathbb{R}^{N_x + 1}$ represents the $j-$th vector of the canonical basis. The existence of the numerical solution to the DQ scheme \eqref{eq:DQ_method} is here addressed by investigating the existence of a fixed point for $\bm{G}^n(\bm{x}).$ The non-negativity assumptions in (\hyperref[ass:1]{A}) imply the global bounds  $\|\bm{\mathcal{M}}_{j}^n(\bm{x})\|_\infty\leq (N_t + 1)\mathsf{R}$ and 
    \begin{equation*}
        \big| \!\left(\bm{\mathcal{M}}_{j}^n\!\left(\bm{x}\right) \cdot \bm{\Psi}^n\right)^\mathsf{T} \!\! \cdot \bm{\Lambda}\big|\leq (N_\lambda+1) \|\bm{\mathcal{M}}_{j}^n(\bm{x})\|_\infty \|\bm{\Psi}^n\|_\infty \|\bm{\Lambda}\|_\infty \leq \mathsf{W}^2\mathsf{R}(N_t+1)(N_\lambda+1),\quad \forall \bm{x}\in \mathbb{R}^{N_x + 1}.
    \end{equation*}  
    It then follows that the restriction of the function $\bm{G}^n(\bm{x})$ to the set $\Omega$ is a continuous self-mapping and a straightforward application of the Brouwer's fixed-point theorem (see, for instance, \cite{Rudin1991}) yields the existence result.\\
    In order to prove the uniqueness of the solution to \eqref{eq:DQ_method} with no limitations on the stepsizes, we apply the findings of  \cite{Uniq_Fix} to the Banach space $\mathbb{R}^{N_x+1}$ ordered by the cone of component-wise non-negative vectors. From the assumptions $(\hyperref[ass:1]{A})$ and the additional hypotheses, $\rho(i(\lambda_l,C_0(x_j),\bm{\Gamma}_{j}^\mathsf{T}\bm{x}))$ is a non-negative and increasing function with respect to $\bm{x}$, whatever the values of $l,$ $j$ and the stepsizes. Consequently, $\bm{G}^n(\bm{x})$ is a compact decreasing operator over $\Omega$ and, from \cite[Theorem 2.1]{Uniq_Fix}, the non empty set $\text{Fix}(\bm{G}^n)=\{\bm{x}\in \Omega \ : \ \bm{x}=\bm{G}^n(\bm{x}) \}$ is directed. Thus, all the hypotheses of \cite[Theorem B]{Uniq_Fix} are satisfied, which completes the proof.
\end{proof}

\begin{remark}
    The unconditional uniqueness of the solution to \eqref{eq:DQ_method} holds under the non-negativity assumptions for the functions $\partial_x f(x)$ and $\varepsilon_{\texttt{B}}(\lambda)-\varepsilon_{\texttt{A}}(\lambda)$. When these requirements are not met,  uniqueness can still be established for sufficiently small stepsizes through the Banach–Caccioppoli theorem (cf. \cite[Theorem 9.23]{Rudin1976}) applied to the function $\bm{G}^n$ in \eqref{eq:Forma_G}. 
\end{remark}

Theorem \ref{thm:DQ_Properties} establishes the DQ discretization \eqref{eq:DQ_method} as a positivity and boundedness preserving numerical method. In comparison to the dynamically consistent NSFD scheme \eqref{eq:NSFD_scheme}, it exhibits a higher level of complexity and does not guarantee the unconditional monotonicity of the numerical solution. However, we will demonstrate that \eqref{eq:DQ_method} achieves high-order convergence that overcomes the accuracy limitations associated with non-standard discretizations. As a first step, we investigate the behavior of the error arising from the approximation of the composition of non-local integral operators through the embedding of quadrature rules with Gregory weights. More specifically, we address the local discretization error of \eqref{eq:DQ_method}, defined as follows
\begin{equation}\label{eq:DQ_local_error}
    \allowdisplaybreaks
    \begin{split}
        \delta_j^n&=\delta(\Delta x, \Delta t, \Delta \lambda; \,  x_j,t_n) \\ 
        &=\displaystyle \scalebox{1.3}{$\displaystyle\int$}_{\!\!\!\!\! 0}^{t_n} \!\! \scalebox{1.3}{$\displaystyle\int$}_{\!\!\!\!\! \lambda_{0}}^{\lambda_{*}} \!\!\! \rho\left(\iota\left(\lambda,\, C_0(x_j), \int_0^{x_j}    c(\xi,\tau)  \ d\xi\right)\right) \ d\lambda \ d \tau \\
        &-\Delta t \Delta \lambda  \Bigg\{ \sum_{m=0}^{n_0-1} \! w_{n m} \Bigg( \, \sum_{i=0}^{n_0-1} \! w_{N_\lambda i} \ \rho \! \left( \! \iota \! \left(\lambda_i, C_0(x_j) , \Delta x \! \sum_{k=0}^{n_0-1} \! w_{jk} \, c(x_k,t_m)+\Delta x \! \sum_{h=n_0}^{j} \! \omega_{j-h} \, c(x_h,t_m) \right)\!\right) \\
        &+ \sum_{l=n_0}^{N_\lambda} \omega_{N_\lambda-l} \ \rho \! \left( \! \iota \! \left(\lambda_l, C_0(x_j) , \Delta x \! \sum_{k=0}^{n_0-1}w_{jk} \, c(x_k,t_m)+\Delta x \! \sum_{h=n_0}^{j}\omega_{j-h} \, c(x_h,t_m) \right)\!\right)\Bigg) \\
        &+ \sum_{p=n_0}^{n} \omega_{n-p} \Bigg( \, \sum_{q=0}^{n_0-1} w_{N_\lambda q} \ \rho \! \left( \! \iota \! \left(\lambda_q, C_0(x_j) , \Delta x \! \sum_{k=0}^{n_0-1}w_{jk} \, c(x_k,t_p)+\Delta x \! \sum_{h=n_0}^{j}\omega_{j-h} \, c(x_h,t_p) \right)\!\right) \\
        &+ \sum_{s=n_0}^{N_\lambda} \omega_{N_\lambda-s} \ \rho \! \left( \! \iota \! \left(\lambda_s, C_0(x_j) , \Delta x \! \sum_{k=0}^{n_0-1}w_{jk} \, c(x_k,t_p)+\Delta x \! \sum_{h=n_0}^{j}\omega_{j-h} \, c(x_h,t_p) \right)\!\right)\Bigg) \Bigg\}, 
    \end{split}
\end{equation}
for $j,n\geq n_0,$ and prove the high order consistency of the DQ scheme \eqref{eq:DQ_method}.

\begin{lemma}\label{lemma:DQ_Cons}
    Assume that the given functions in \eqref{eq:VFIE_model}, describing problem \eqref{eq:continuous_model}, satisfy $f,c^0\in C^{n_0+1}([0,L]),$ $I\in C^{n_0+1}([\lambda_{0},\lambda_{*}]),$ $\rho\in C^{n_0+1}(\mathbb{R}^+_0).$ Then the DQ method \eqref{eq:DQ_method} is consistent with \eqref{eq:VFIE_model}, of order $n_0+1$. 
\end{lemma}

\begin{proof}
    Let $L=N_x \Delta x,$ $T=N_t \Delta t$ and $\lambda_{*} =\lambda_{0}+ N_\lambda \Delta \lambda,$ with $N_x,$ $N_t$ and $N_\lambda$ positive integers. Because of to the regularity assumptions and by applying the mean value theorem, the local discretization error in \eqref{eq:DQ_local_error} satisfies 
    \begin{equation}\label{eq:partial_bound_DQ_cons}
        \begin{split}
             |\delta_j^n|&\leq |\varphi (\Delta t;  \, x_j, t_n)|+\Delta t\left( \sum_{m=0}^{n_0-1} w_{n m} |\varrho(\Delta \lambda;  \, x_j, t_{m})| + \sum_{p=n_0}^{n} \omega_{n-p} |\varrho(\Delta \lambda;  \, x_j, t_{p})| \right) \\
             &+\kappa_1 \Delta t \Delta \lambda 
             \left(\sum_{m=0}^{n_0-1} w_{n m}|\chi(\Delta x;  \, x_j, t_{m})|+\sum_{p=n_0}^{n}  \omega_{n-p}|\chi(\Delta x;  \, x_j, t_{p})|\right)
             \left( \sum_{i=0}^{n_0-1} w_{N_\lambda i}  + \sum_{l=n_0}^{N_\lambda} \omega_{N_\lambda-l} \right),
        \end{split}
    \end{equation}
    for $n_0\leq t\leq N_t,$ $n_0\leq j\leq N_x,$ where $\kappa_1\in \mathbb{R}^+$ depends on the bounds on the known functions and their derivatives but not on the stepsizes and 
    \begin{equation*}
        \begin{split}
            \varphi (\Delta t;  \, x_j, t_n)&= \scalebox{1.3}{$\displaystyle\int$}_{\!\!\!\!\! 0}^{t_n} \!\!\! \scalebox{1.3}{$\displaystyle\int$}_{\!\!\!\!\! \lambda_{0}}^{\lambda_{*}} \!\!\! \rho \! \left( \! \iota \! \left(\lambda,\, C_0(x_j), \int_0^{x_j} \!\! c(\xi,\tau)  \ d\xi\right)\right) \ d\lambda \ d \tau \\
            &- \Delta t\sum_{m=0}^{n_0-1} w_{n m} \scalebox{1.3}{$\displaystyle\int$}_{\!\!\!\!\! \lambda_0}^{\lambda_*} \!\!\! \rho \! \left( \! \iota \! \left(\lambda, C_0(x_j) , \int_0^{x_j} \!\! c(\xi,t_m)  \ d\xi \right) \right) \ d\lambda \\
            &- \Delta t \sum_{p=n_0}^{n} \omega_{n-p} \scalebox{1.3}{$\displaystyle\int$}_{\!\!\!\!\! \lambda_0}^{\lambda_*} \!\!\! \rho \! \left( \! \iota \! \left(\lambda, C_0(x_j) , \int_0^{x_j} \!\!   c(\xi,t_p)  \ d\xi \right) \right) \ d\lambda, 
            \qquad \qquad \begin{array}{l}
            n=n_0,\dots, N_t,  \\
            j \, =n_0,\dots, N_x, 
        \end{array}
        \end{split}
    \end{equation*}
    \begin{equation*}
        \begin{split}
            \varrho(\Delta \lambda;  \, x_j, t_{\nu})&=\scalebox{1.3}{$\displaystyle\int$}_{\!\!\!\!\! \lambda_0}^{\lambda_*} \!\!\! \rho \! \left( \! \iota \! \left(\lambda, C_0(x_j) , \int_0^{x_j} \!\!   c(\xi,t_\nu)  \ d\xi \right) \right) \ d\lambda \\
            & - \Delta \lambda  \sum_{i=0}^{n_0-1} w_{N_\lambda i} \ \rho \! \left( \! \iota \! \left(\lambda_i, C_0(x_j) , \int_0^{x_j} \!\! c(\xi,t_\nu)  \ d\xi \right) \right) \\
            & - \Delta \lambda \sum_{l=n_0}^{N_\lambda} \omega_{N_\lambda-l} \ \rho \! \left( \! \iota \! \left(\lambda_l, C_0(x_j) , \int_0^{x_j} \!\! c(\xi,t_\nu)  \ d\xi \right) \right), 
            \qquad \qquad \qquad \quad \nu=n_0,\dots,n,
        \end{split}
    \end{equation*}
    \begin{equation*}
        \chi(\Delta x;  \, x_j, t_{\nu})=\int_0^{x_j} \!\!   c(\xi,t_\nu)  \ d\xi-\Delta x \left(\sum_{k=0}^{n_0-1}w_{jk} \, c(x_k,t_\nu)+ \sum_{h=n_0}^{j}\omega_{j-h} \, c(x_h,t_\nu)\right), \phantom{\qquad  \nu=n_0,\dots,n,}
    \end{equation*}
    are the errors related to separate time, frequency and space integral approximations, respectively. The properties of the employed $(n_0-1)$-th Gregory rule implies that
    \begin{equation*}
        |\varphi (\Delta t;  \, x_j, t_n)|\leq \kappa_2 \Delta t^{n_0+1}, \qquad |\varrho (\Delta \lambda;  \, x_j, t_n)|\leq \kappa_3 \Delta \lambda^{n_0+1}, \qquad |\chi (\Delta x;  \, x_j, t_n)|\leq \kappa_4 \Delta x^{n_0+1},
    \end{equation*}
    with $\kappa_2, \kappa_3$ and $\kappa_4$ positive constants and therefore, from \eqref{eq:partial_bound_DQ_cons},
    \begin{equation*}
        \begin{split}
             \max_{\substack{n=0,\dots, N_t \\ j=0,\dots, N_x}}|\delta(\Delta x, \Delta t, \Delta \lambda; \,  x_j,t_n)|&\leq \kappa \left(\Delta t^{n_0+1}+\Delta x^{n_0+1}+\Delta \lambda^{n_0+1}\right)\leq \kappa \left(\Delta t+\Delta x+\Delta \lambda\right)^{n_0+1},
        \end{split}
    \end{equation*}
    that yields the outcome.
\end{proof}

A quite standard analysis leads to the following convergence result on bounded intervals.
\begin{theorem}\label{thm:DQ_Convergence}
    Assume that the given functions in \eqref{eq:VFIE_model}, describing problem \eqref{eq:continuous_model}, satisfy $f,c^0\in C^{n_0+1}([0,L]),$ $I\in C^{n_0+1}([\lambda_{0},\lambda_{*}]),$ $\rho\in C^{n_0+1}(\mathbb{R}^+_0).$ Let $\{c^n_j\}$ be the approximations of the continuous solution to \eqref{eq:continuous_model} computed by the DQ scheme \eqref{eq:DQ_method} with $\Delta x=L/N_x,$ $\Delta t=T/N_t,$ $\Delta \lambda=(\lambda_{*}-\lambda_{0})/N_\lambda$ and $N_x,N_t,N_\lambda\in \mathbb{N}.$ 
    If the starting errors satisfy
    \begin{equation}\label{eq:starting_errors}
        |\eta_k^m|=|c(x_k,t_m)-c^m_k|=\mathcal{O}(\Delta x+\Delta t+\Delta \lambda)^{n_0+1}, \quad k=0,\dots,n_0-1, \quad m=0,\dots,n_0-1,
    \end{equation}
    then the DQ method \eqref{eq:DQ_method} is convergent of order $n_0+1.$
\end{theorem}

\begin{proof}
    The bounds on the continuous and the numerical solutions (cf. Theorem \ref{thm:DQ_Properties}) imply that 
    \begin{equation*}
        \left|\mathsf{e}_j^n\right|\leq \left| \log\left(\dfrac{c(x_j,t_n)}{c^0(x_j)}\right)-\log\left(\dfrac{c_j^n}{c^0(x_j)}\right) \right|, \qquad j=n_0,\dots, N_x, \qquad \ n=n_0,\dots, N_t,
    \end{equation*}
    where $\mathsf{e}_j^n=\mathsf{e}(\Delta x, \Delta t, \Delta \lambda; \,  x_j,t_n)=c(x_j,t_n)-c_j^n,$ is the global approximation error of the DQ method \eqref{eq:DQ_method}. Therefore, evaluating \eqref{eq:VFIE_model} at the point $(x_j,t_n)$ and subtracting \eqref{eq:DQ_method} from it, leads to 
    \begin{equation*}
            \left|\mathsf{e}_j^n\right|\leq \kappa_1  \left|\delta_j^n\right| + \kappa_2  \mathsf{W} \lambda_* \, \Delta x \, \Delta t \,  \left(\sum_{m=0}^{n_0-1}\sum_{p=n_0}^{n}  \left(\sum_{k=0}^{n_0-1}  w_{jk}(|\eta_k^m|+|\mathsf{e}_k^p|)+  \sum_{h=n_0}^{j}  \omega_{j-h}(|\eta_h^m|+|\mathsf{e}_h^p|)\right)  \right),
    \end{equation*}
    for $j=n_0,\dots, N_x$ and $n=n_0,\dots, N_t$, where $\delta_j^n$ denotes the DQ local error in \eqref{eq:DQ_local_error}, $\mathsf{W}=\max\{\tilde{w}, \tilde{\omega}\}$ is related to the weights in \eqref{eq:weights_bounds} and $\kappa_i>0,$ $i\in\{1,2\},$ is a constant depending on the bounds on the known functions and their derivatives. Hence, for a sufficiently small temporal stepsize, the inequality
    \begin{equation*}
        \begin{split}
            \max_{j=0,\dots,N_x}|\mathsf{e}_j^n|\leq \dfrac{\kappa_1 \max_{j=0,\dots,N_x}|\delta_j ^n|}{1-\Delta t \, \kappa_2 \mathsf{W}^2 L \lambda_*}+ \dfrac{\Delta t \, \kappa_2 \mathsf{W}^2 L \lambda_*}{1-\Delta t \, \kappa_2 \mathsf{W}^2 L \lambda_*} \sum_{m=0}^{n-1} \max_{j=0,\dots,N_x}|\mathsf{e}_j^m|, \qquad \quad n=n_0,\dots, N_t.
        \end{split}
    \end{equation*}
    holds true. The Gronwall-type discrete inequality in \cite[Theorem 7.1]{Linz_1985} then yields
    \begin{equation*}
        \begin{split}
            \max_{\substack{n=0,\dots, N_t \\ j=0,\dots, N_x}}|\mathsf{e}_j^n|\leq \left(\dfrac{\kappa_1 + \Delta t \, \kappa_2 \mathsf{W}^2 L \lambda_*}{1-\Delta t \, \kappa_2 \mathsf{W}^2 L \lambda_*}  \right)\exp\left\{\dfrac{\kappa_2 \mathsf{W}^2 LT\lambda_*}{1-\Delta t \, \kappa_2 \mathsf{W}^2 L \lambda_*}\right\}\left(\max_{\substack{n=0,\dots, N_t \\ j=0,\dots, N_x}}|\delta_j ^n|+\max_{\substack{n=0,\dots, N_t \\ j=0,\dots, N_x}}|\eta_j ^n|\right),
        \end{split}
    \end{equation*}
    so that, from the consistency findings of Lemma \ref{lemma:DQ_Cons} and the hypothesis \eqref{eq:starting_errors} on the starting errors, we get the result.
\end{proof}

\section{A predictor-corrector approach}\label{sec:PC}

The NSFD scheme \eqref{eq:NSFD_scheme} and the DQ method \eqref{eq:DQ_method}, defined in Sections \ref{sec:NSFD} and \ref{sec:DQ} respectively, represent two strategies for the dynamically consistent numerical simulation of the integro-differential model \eqref{eq:continuous_model}. Both ensure positivity regardless of the discretization stepsizes. The linearly implicit NSFD scheme additionally preserves monotonicity, but is only first-order accurate, whereas the higher-order DQ method requires greater computational effort. Here, to retain the advantages of the aforementioned approaches, we combine them within a Predictor-Corrector (PC) discretization. We consider, for $\Delta x, \Delta t, \Delta \lambda\in \mathbb{R}^+,$ the uniform meshes $\{x_j=j \Delta x, \ j=0,1,\dots \},$ $\{t_n=n \Delta t, \ n=0,1,\dots  \},$ and $\{\lambda_l=\lambda_{0}+l \Delta \lambda, \ l=0,1,\dots, N_\lambda \}.$ Given the the positive initial values $c_j^0=c^0(x_j)=c(x_j,0),$ we  define the following PC integrator
\begin{equation}\label{eq:PC}
    \begin{cases}
         p_j^{n}=c_j^{n-1}\left\{\displaystyle 1+ \varphi(\Delta t) \, \Delta \lambda \, f(x_j) \sum_{l=0}^{N_\lambda-1}\rho\left(\iota\left(\lambda_l,\, C_0(x_j) ,\, \Delta x \sum_{r=0}^{j-1}c_r^{n-1} \right)\right)\right\}^{-1}, \\
         \displaystyle\beta_{j}^{l}=\rho\!\left( \! \iota \! \left(\lambda_l,\, C_0(x_j) ,\, \dfrac{\Delta x}{2} \!\left(c_0^{n-1}+2\sum_{k=1}^{j-1}c_k^{n-1}+c_j^{n-1} \right) \right)\!\right),                                   \\
         \displaystyle\gamma_{j}^{l}=\rho\!\left( \! \iota \! \left(\lambda_l,\, C_0(x_j) ,\, \dfrac{\Delta x}{2} \!\left(p_0^{n}+2\sum_{k=1}^{j-1}p_k^{n}+p_j^{n} \right) \right)\!\right), \qquad \qquad \qquad \qquad \qquad \qquad l=0,\dots,N_\lambda, \\
         c^{n}_j   \displaystyle=c^{n-1}_j  \exp\left\{ \! -\dfrac{\Delta t}{2} \dfrac{\Delta \lambda}{2}  f(x_j) \! \left( \!\beta_j^{0}+\gamma_{j}^{0}+ 2\!\!\sum_{l=1}^{N_{\lambda}-1}\!\left(\beta_j^{l}+\gamma_{j}^{l}\right)+ \beta_j^{N_\lambda}+\gamma_{j}^{N_\lambda} \! \right) \!   \right\}, 
    \end{cases}
\end{equation}
where the function $\varphi$ satisfies \eqref{eq:prop_phi} and $c_j^n\approx c(x_j,t_n),$ for $n\geq 1$ and $j\geq 0.$

The PC method \eqref{eq:PC} relies on the NSFD scheme \eqref{eq:NSFD_scheme} to predict the solution, followed by a correction phase based on the DQ scheme \eqref{eq:DQ_method} with $n_0=1$ and trapezoidal weights. Consequently, it inherits the advantages of both discretizations, as proved by the following theorem. 

\begin{theorem}\label{thm:PC_Dynamic_Consistency}
    Consider the discrete equations \eqref{eq:PC} under the assumptions (\hyperref[ass:1]{A}) and \eqref{eq:prop_phi}. Then, independently of $\Delta x, \Delta t, \Delta \lambda\in \mathbb{R}^+$ and for each $j\in \mathbb{N}_0,$ the sequences $\{ p_j^n \}_{n\in \mathbf{N}_0}$ and $\{ c_j^n \}_{n\in \mathbf{N}_0}$ are positive and bounded from above. Furthermore, $\{ c_j^n \}_{n\in \mathbf{N}_0}$ is monotonically decreasing with respect to $n$.
\end{theorem}
\begin{proof}
     Standard induction arguments yield the positivity of $\{ p_j^n \}_{n\in \mathbf{N}_0}$ and $\{ c_j^n \}_{n\in \mathbf{N}_0},$ independently of $j.$ Furthermore, due to the assumptions (\hyperref[ass:1]{A}), the terms $\beta_{j}^{l}$ and $\gamma_{j}^{l}$ are unconditionally non-negative. Therefore, from the last equation in \eqref{eq:PC}, $c_j^n<c_j^{n-1}<c_j^0$ for all $n\geq 1,$ which completes the proof.
\end{proof}

\begin{theorem}\label{thm:PC_conv}
    Assume that the known functions in \eqref{eq:continuous_model}  satisfy $f,c^0\in C^{2}([0,L]),$ $I\in C^{2}([\lambda_{0},\lambda_{*}]),$ $\rho\in C^{2}(\mathbb{R}^+_0).$ Let $\{c^n_j\}$ be the approximations of the continuous solution to \eqref{eq:continuous_model} computed by the PC scheme \eqref{eq:PC} with $\Delta x=L/N_x,$ $\Delta t=T/N_t,$ $\Delta \lambda=(\lambda_{*}-\lambda_{0})/N_\lambda$ and $N_x,N_t,N_\lambda\in \mathbb{N}.$ Denote by $\mathsf{e}(\Delta x, \Delta t, \Delta \lambda; \,  x_j,t_n)=c(x_j,t_n)-c_j^n$ the PC approximation error. Then, there exists a constant $\kappa\in\mathbb{R}^+$ such that 
    \begin{equation*}
        \max_{\substack{n=0,\dots, N_t \\ j=0,\dots, N_x}}|\mathsf{e}(\Delta x, \Delta t, \Delta \lambda; \,  x_j,t_n)|\leq  \kappa (\Delta x+\Delta t+\Delta \lambda)^2.
    \end{equation*}
    Furthermore, the predictor-corrector method \eqref{eq:PC} is convergent of order $2.$
\end{theorem}
\begin{proof}
    The outcomes of Lemmas \ref{lemma:NSFD_Cons} and \ref{lemma:DQ_Cons} assure that for the local discretization errors  $\prescript{P}{}{\delta}_j^n$ and $\prescript{C}{}{\delta}_j^n$ of the predictor and the corrector scheme, respectively, the inequalities 
    \begin{equation*}
        \max_{\substack{n=0,\dots, N_t \\ j=0,\dots, N_x}}|\prescript{P}{}{\delta}_j^n|\leq  \kappa_1 (\Delta x+\Delta t+\Delta \lambda), \qquad \qquad  \max_{\substack{n=0,\dots, N_t \\ j=0,\dots, N_x}}|\prescript{C}{}{\delta}_j^n|\leq  \kappa_2 (\Delta x+\Delta t+\Delta \lambda)^2,
    \end{equation*}
    hold true, with $\kappa_1$ and $\kappa_2$ positive constants. The second order convergence then comes from a straightforward extension of \cite[Theorem 2]{PC_Volterra}.  
\end{proof}

The findings of Theorems \ref{thm:PC_Dynamic_Consistency} and \ref{thm:PC_conv} establish the dynamical consistency and the quadratic convergence of the explicit PC numerical method \eqref{eq:PC}, respectively.
Therefore, when compared to the NSFD method \eqref{eq:NSFD_scheme}, the PC scheme is expected to be more accurate. On the other hand, in contrast to the DQ approach \eqref{eq:DQ_method}, it unconditionally preserves the monotonicity of the solution and is less demanding to implement, given its more straightforward structure. For this reason, in what follows, the DQ method will be employed only for $n_0>1$ (i.e. for convergence orders greater than $2$), while the PC integrator \eqref{eq:PC} will be preferred for a second-order approximation of the solution to \eqref{eq:continuous_model}.
We refer to Section \ref{sec: Numerical_Simulations} for a comprehensive comparative analysis of the NSFD, DQ and PC discretizations.

\section{Numerical Simulations}\label{sec: Numerical_Simulations}
In this section, we present numerical experiments to illustrate the theoretical properties discussed in the preceding sections and to compare the effectiveness of the proposed methods. Additionally, we explore two realistic applications of the model \eqref{eq:continuous_model} pertaining to serotonin photoactivation and cadmium yellow photodegradation. All simulations are conducted using MATLAB R2023a on an Intel(R) Core(TM) i9-14900KF processor operating at 3200 MHz.

\medskip

\textbf{Test 1}\label{defn: Test_1} - For our first test we consider the problem defined in \eqref{eq:continuous_model} with 
\begin{equation}\label{eq:Test_1}
    \begin{split}
        c^0(x)&=\exp\left\{-\dfrac{x^2}{5}\right\}, \quad C^0(x)= \dfrac{\sqrt{5\pi}}{2}\erf{\left\{\dfrac{x}{\sqrt{5}}\right\}}, \quad  f(x)=2-x, \quad  L=1, \quad T=1, \\
        I(\lambda)&=\lambda^2, \quad  \varepsilon_\texttt{A}(\lambda)=5+\lambda, \quad \varepsilon_\texttt{B}(\lambda)=\lambda , \quad  \lambda_0=0, \quad  \lambda^*=1, \quad  a_1=a_2=1, \quad \mu=0.1.  
    \end{split}
\end{equation}
The functions and parameters in \eqref{eq:Test_1} serve as an illustrative example to assess the mathematical properties and evaluate the performances of the proposed schemes. The numerical solution of \hyperref[defn: Test_1]{Test 1} computed using the fourth-order, three-step method \eqref{eq:DQ_method} with II Gregory rule and $\Delta x=\Delta t=\Delta \lambda=2^{-9},$ is shown in Figure \ref{fig:Ar1_Reference}. This solution, denoted by $\bm{\mathcal{C}}=\{\mathcal{C}_j^n\}\in \mathbb{R}^{N_x \times N_t}$ hereafter, is used as a benchmark to compute the mean space-time approximation errors and the experimental order of convergence  defined as follows\begin{equation}\label{eq:Error_mean_formulation}
        E(\Delta z, \Delta t, \Delta \lambda)=\dfrac{\sum_{j=0}^{N_z}\sum_{n=0}^{N_t} \left|c_j^n-\mathcal{C}_j^n\right|}{N_xN_t}, \qquad \hat{p}=\log_2\left(\dfrac{E(\Delta z, \Delta t, \Delta \lambda)}{E(\frac{\Delta z}{2}, \frac{\Delta t}{2}, \frac{\Delta \lambda}{2})}\right).
\end{equation}

\begin{figure}[htbp]
  \centering
  \hspace{-1.6 cm}\begin{minipage}{0.45\textwidth}
    \centering
    \includegraphics[width=1.26\linewidth]{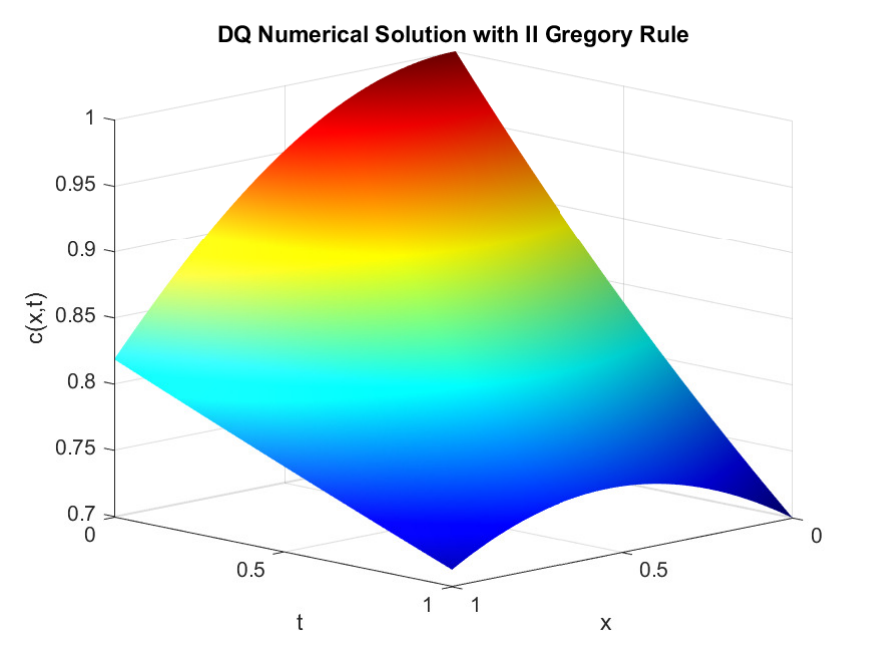}
  \end{minipage}\hspace{0.05\textwidth} 
  \begin{minipage}{0.45\textwidth}
    \centering
    \includegraphics[width=1.26\linewidth]{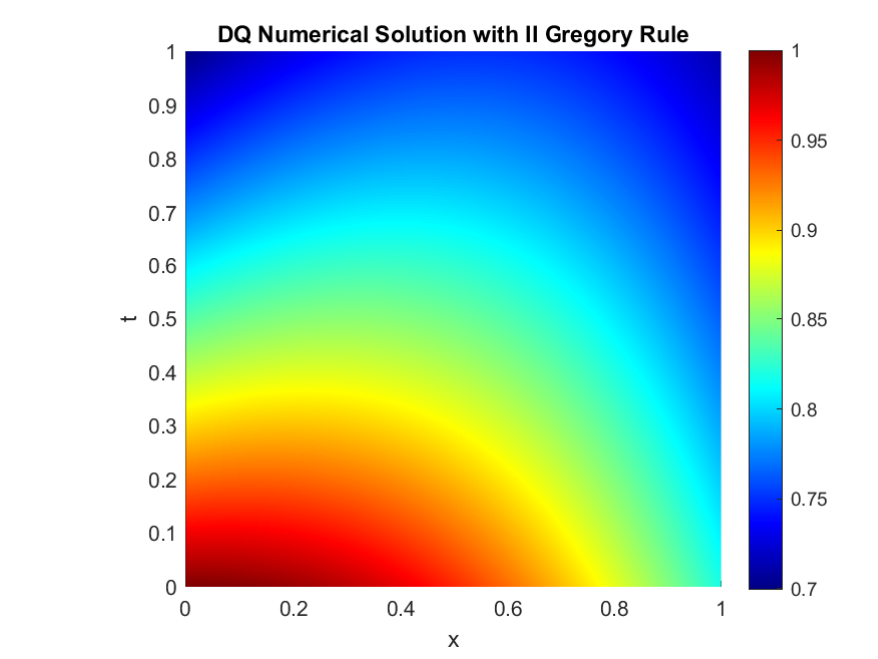}
  \end{minipage}
  \caption{Reference numerical solution of \hyperref[defn: Test_1]{Test 1} obtained with the DQ scheme \eqref{eq:DQ_method}, $n_0=3$ and $\Delta x=\Delta t=\Delta \lambda=2^{-9}.$}
  \label{fig:Ar1_Reference}
\end{figure}

As a first step, we analyze the simulation outcomes of the NSFD numerical method \eqref{eq:NSFD_scheme} endowed with the following denominator functions
\begin{equation}\label{eq:Denominator_Functions}
    \phi_1(\Delta t)=\Delta t,  \qquad \qquad \phi_2(\Delta t)=\Delta t(1+\gamma \Delta t),
\end{equation}
where $\gamma=\max_{[0,L]}f(x)=2.$ In both cases, the conditions \eqref{eq:prop_phi} are satisfied and the NSFD scheme \eqref{eq:NSFD_scheme} ensures positivity and monotonicity of the solutions. Moreover, it demonstrates experimental linear convergence, in compliance with the theoretical findings of Theorems \ref{thm:NSFD_Dynamic_Consistency} and \ref{thm:NSFD_conv}. The approximation errors and the experimental order of convergence, computed as detailed in \eqref{eq:Error_mean_formulation} for different values of the stepsizes, are reported in Table \ref{tab:Test1_Err_I_ORDER} and displayed in Figure \ref{fig:Exp_Order_1}. Notably, the choice of the denominator function \(\phi_2(\Delta t)\) results in improved overall accuracy compared to the standard option \(\phi_1(\Delta t)\). We also tested the function $\phi_3(\Delta t)=\gamma^{-1}(1-e^{-\gamma \Delta t})$ introduced in \cite{Manning2006IntroductionTN,NSFD_INB} but the corresponding results, here omitted for the sake of brevity, were slightly less accurate than those obtained with $\phi_1(\Delta t).$ 

The following dynamically consistent scheme
\begin{equation}\label{eq:RQ}
    c_j^{n+1}=c_j^{n} \exp\left\{-\Delta t \Delta \lambda f(x_j) \sum_{l=0}^{N_\lambda -1} \rho\!\left( \! \iota \! \left(\lambda_l,\, C_0(x_j) ,\, \Delta x \! \sum_{r=0}^{j-1}c_r^{n}\right)\!\right) \right\},
\end{equation}
introduced in \cite{Ceseri_Natalini_Pezzella} and based upon a Rectangular Quadrature (RQ) discretization of \eqref{eq:VFIE_model}, represents a viable option for the simulation of \eqref{eq:continuous_model}. It unconditionally preserves positivity and monotonicity and can be interpreted as an explicit and linearly convergent, simplified variant of the DQ approach \eqref{eq:DQ_method}. The error plot in Figure \ref{fig:Exp_Order_1} and the work precision diagram in Figure \ref{fig:WPD_1} highlight the competitiveness of the RQ scheme \eqref{eq:RQ}, which slightly outperforms the NSFD approach when $\phi_1(\Delta t)$ is used as the denominator function. However, the adoption of $\phi_2(\Delta t)$ in the NSFD method yields superior accuracy and efficiency, surpassing the RQ discretization as well.

\begin{table}[htbp] \center 
\begin{tabular}{ccc|ccc|ccc}
\hline\noalign{\smallskip}
\multicolumn{3}{c|}{NSFD$_{\Delta t}$ } & \multicolumn{3}{|c}{RQ} & \multicolumn{3}{|c}{NSFD$_{\Delta t(1+\gamma\Delta t)}$}  \\
\noalign{\smallskip}\hline\noalign{\smallskip}
$\vartheta$ & $E(\Delta x, \Delta t, \Delta \lambda)$ & $\hat{p}$ & $\vartheta$ & $E(\Delta x, \Delta t, \Delta \lambda)$ & $\hat{p}$ & $\vartheta$ & $E(\Delta x, \Delta t, \Delta \lambda)$ & $\hat{p}$ \\
$2^{-2}$ & $2.63 \cdot 10^{-2}$ & --- & $2^{-2}$ & $2.45 \cdot 10^{-2}$ & --- & $2^{-2}$ & $7.77 \cdot 10^{-3}$ & --- \\
$2^{-3}$ & $1.36 \cdot 10^{-2}$ & $0.96$ & $2^{-3}$ & $1.24 \cdot 10^{-2}$ & $0.99$ & $2^{-3}$ & $6.89 \cdot 10^{-3}$ & $0.17$ \\
$2^{-4}$ & $6.91 \cdot 10^{-3}$ & $0.97$ & $2^{-4}$ & $6.22 \cdot 10^{-3}$ & $0.99$ & $2^{-4}$ & $4.26 \cdot 10^{-3}$ & $0.69$ \\
$2^{-5}$ & $3.49 \cdot 10^{-3}$ & $0.99$ & $2^{-5}$ & $3.12 \cdot 10^{-3}$ & $1.00$ & $2^{-5}$ & $2.35 \cdot 10^{-3}$ & $0.86$ \\
$2^{-6}$ & $1.75 \cdot 10^{-3}$ & $0.99$ & $2^{-6}$ & $1.56 \cdot 10^{-3}$ & $1.00$ & $2^{-6}$ & $1.23 \cdot 10^{-3}$ & $0.93$ \\
$2^{-7}$ & $8.78 \cdot 10^{-4}$ & $1.00$ & $2^{-7}$ & $7.81 \cdot 10^{-4}$ & $1.00$ & $2^{-7}$ & $6.27 \cdot 10^{-4}$ & $0.97$ \\
\noalign{\smallskip} \hline
\end{tabular}
\smallskip
\caption{Approximation errors and experimental convergence rates for the NSFD and RQ numerical solutions of \hyperref[defn: Test_1]{Test 1}, with $\Delta x = \Delta t = \Delta \lambda = \vartheta$.}
\label{tab:Test1_Err_I_ORDER}       
\end{table}

\begin{figure}[htbp]
  \centering
  \begin{minipage}{0.45\textwidth}
    \centering
    \includegraphics[width=0.85\linewidth]{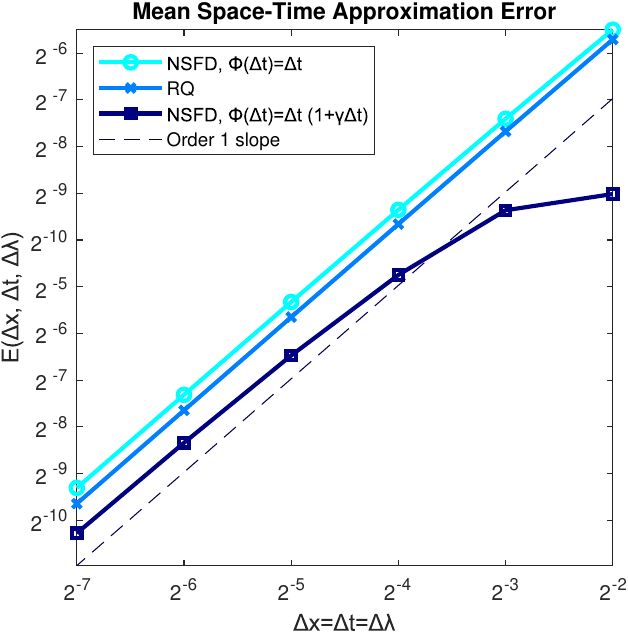}
    \caption{Base-2 logarithmic plot of the error as a function of the stepsizes for the NSFD method \eqref{eq:NSFD_scheme} and the RQ scheme \eqref{eq:RQ}.}
    \label{fig:Exp_Order_1}
  \end{minipage}\hspace{0.05\textwidth} 
  \begin{minipage}{0.45\textwidth}
    \centering
    \includegraphics[width=0.85\linewidth]{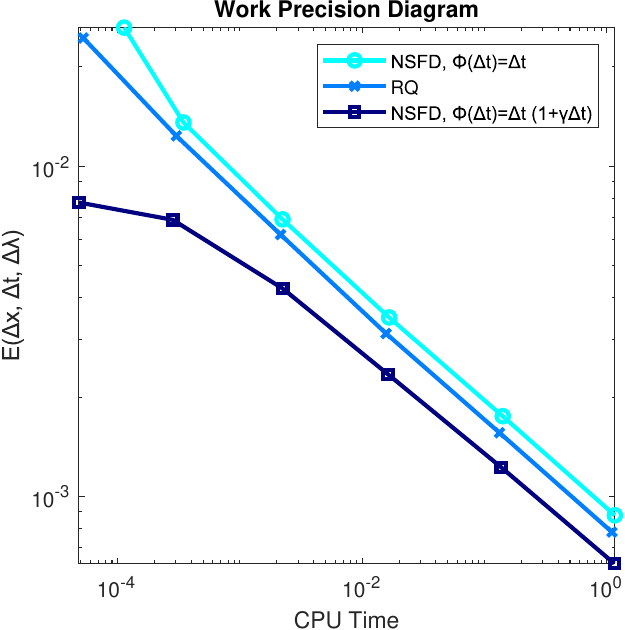}
    \caption{Logarithmic plot of the error as a function of the execution time for the NSFD method \eqref{eq:NSFD_scheme} and the RQ scheme \eqref{eq:RQ}. }
    \label{fig:WPD_1}
  \end{minipage}
\end{figure}

We investigate the PC scheme \eqref{eq:PC} and the DQ method \eqref{eq:DQ_method} as more involved and accurate integrators for \hyperref[defn: Test_1]{Test 1}. Specifically, we consider \eqref{eq:DQ_method} with I and II Gregory weights ($n_0\in\{2,3\}$) and generate the required starting values satisfying \eqref{eq:starting_errors} using the PC approach and the I Gregory-DQ discretization, respectively.  Furthermore, at each time step, the fixed point of \eqref{eq:Forma_G} is computed using the trust-region dogleg algorithm from the \texttt{fsolve} MATLAB routine, with step and function value tolerances set to $10^{-14}$. The results of the experiments, reported in Table \ref{tab:Test1_Err_High_ORDER} and Figure \ref{fig:Exp_Order_High}, validate the theoretical predictions of Theorems \ref{thm:DQ_Convergence} and \ref{thm:PC_conv}, demonstrating the high-order convergence of the proposed discretizations (quadratic for the PC scheme, of order $n_0+1$ for the DQ method). The errors, also in this case, are computed following \eqref{eq:Error_mean_formulation}. From the work precision diagram in Figure \ref{fig:WPD_High} it is clear that, although each iteration of the implicit DQ method is more demanding than that of the explicit PC scheme, the former proves to be overall more efficient in terms of the accuracy-computational cost trade-off.

\begin{figure}[htbp]
  \centering
  \begin{minipage}{0.45\textwidth}
    \centering
    \includegraphics[width=0.85\linewidth]{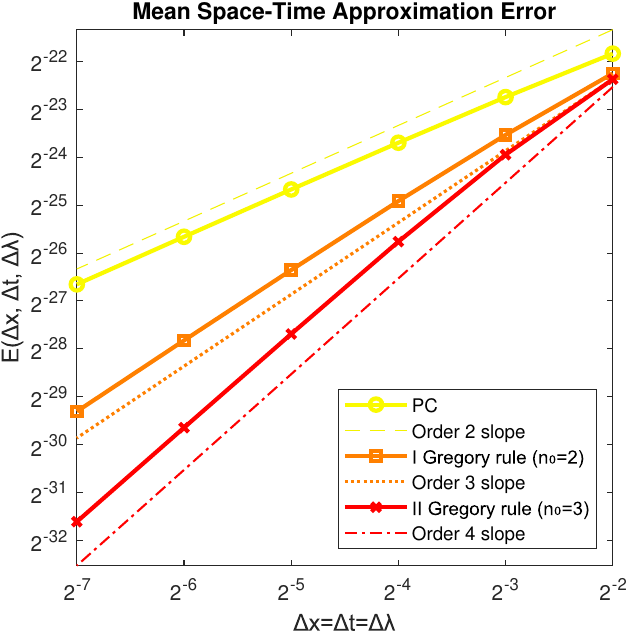}
    \caption{Base-2 logarithmic plot of the error as a function of the stepsizes for the PC method \eqref{eq:PC} and the DQ scheme \eqref{eq:DQ_method}.}
    \label{fig:Exp_Order_High}
  \end{minipage}\hspace{0.05\textwidth} 
  \begin{minipage}{0.45\textwidth}
    \centering
    \includegraphics[width=0.85\linewidth]{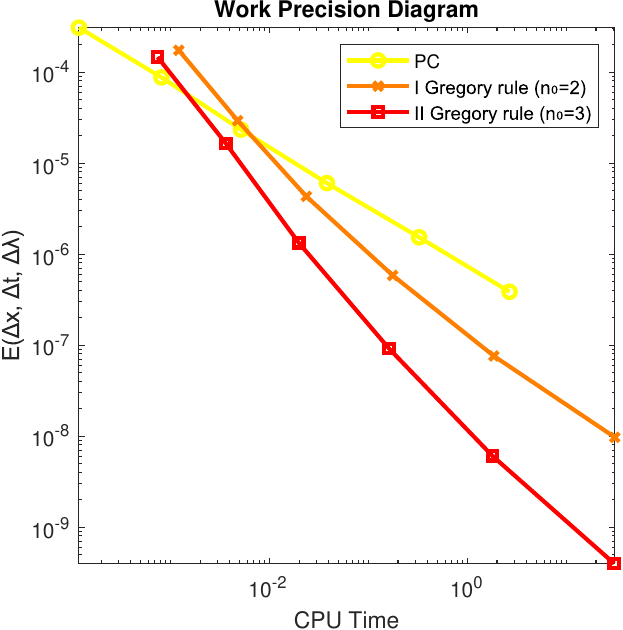}
    \caption{Logarithmic plot of the error as a function of the execution time for the PC method \eqref{eq:PC} and the DQ scheme \eqref{eq:DQ_method}.}
    \label{fig:WPD_High}
  \end{minipage}
\end{figure}

\begin{table}[htbp] \center 
\begin{tabular}{ccc|ccc|ccc}
\hline\noalign{\smallskip}
\multicolumn{3}{c}{PC} & \multicolumn{3}{|c}{DQ-I Gregory} {\small$(n_0=2)$} \phantom{c} & \multicolumn{3}{|c}{DQ-II Gregory} {\small$(n_0=3)$} \phantom{c} \\
\noalign{\smallskip}\hline\noalign{\smallskip}
$\vartheta$ & $E(\Delta x, \Delta t, \Delta \lambda)$ & $\hat{p}$ & $\vartheta$ & $E(\Delta x, \Delta t, \Delta \lambda)$ & $\hat{p}$ & $\vartheta$ & $E(\Delta x, \Delta t, \Delta \lambda)$ & $\hat{p}$ \\
$2^{-2}$ & $3.07 \cdot 10^{-4}$ & --- & $2^{-2}$ & $1.74 \cdot 10^{-4}$ & --- & $2^{-2}$ & $1.46 \cdot 10^{-4 \phantom{0}}$ & --- \\
$2^{-3}$ & $8.74 \cdot 10^{-5}$ & $1.82$ & $2^{-3}$ & $2.92 \cdot 10^{-5}$ & $2.57$ & $2^{-3}$ & $1.65 \cdot 10^{-5 \phantom{0}}$ & $3.15$ \\
$2^{-4}$ & $2.33 \cdot 10^{-5}$ & $1.91$ & $2^{-4}$ & $4.30 \cdot 10^{-6}$ & $2.77$ & $2^{-4}$ & $1.33 \cdot 10^{-6 \phantom{0}}$ & $3.63$ \\
$2^{-5}$ & $6.02 \cdot 10^{-6}$ & $1.95$ & $2^{-5}$ & $5.84 \cdot 10^{-7}$ & $2.88$ & $2^{-5}$ & $9.15 \cdot 10^{-8 \phantom{0}}$ & $3.86$ \\
$2^{-6}$ & $1.53 \cdot 10^{-6}$ & $1.98$ & $2^{-6}$ & $7.60 \cdot 10^{-8}$ & $2.94$ & $2^{-6}$ & $6.07 \cdot 10^{-9 \phantom{0}}$ & $3.91$ \\
$2^{-7}$ & $3.85 \cdot 10^{-7}$ & $2.00$ & $2^{-7}$ & $9.71 \cdot 10^{-9}$ & $2.97$ & $2^{-7}$ & $3.98 \cdot 10^{-10}$ & $3.93$ \\
\noalign{\smallskip} \hline
\end{tabular}
\smallskip
\caption{Approximation errors and experimental convergence rates for the PC and DQ numerical solutions of \hyperref[defn: Test_1]{Test 1}, with $\Delta x = \Delta t = \Delta \lambda = \vartheta$.}
\label{tab:Test1_Err_High_ORDER}       
\end{table}

In order to experimentally highlight the positivity-preserving property of the DQ method \eqref{eq:DQ_method}, as established in Theorem \ref{thm:DQ_Properties}, we present in Figure \ref{fig:Thm_Verification_DQ1} the DQ numerical solution obtained with first Gregory rule $(n_0=2)$ and the corresponding theoretical lower bound given in \eqref{eq:DQ_bound_inferiore}, for various stepsize values. Although not sharp, such bound confirms the scheme's unconditional positivity. The outcomes of the DQ-II Gregory rule $(n_0=3)$ simulations, involving different weights-dependent values of $\bar{c}(\Delta x, \Delta t, \Delta \lambda)$, exhibit similar characteristics but are here excluded for conciseness.

\begin{figure}[htbp]
\hspace{-0.15\linewidth}\includegraphics[width=1.3\linewidth]{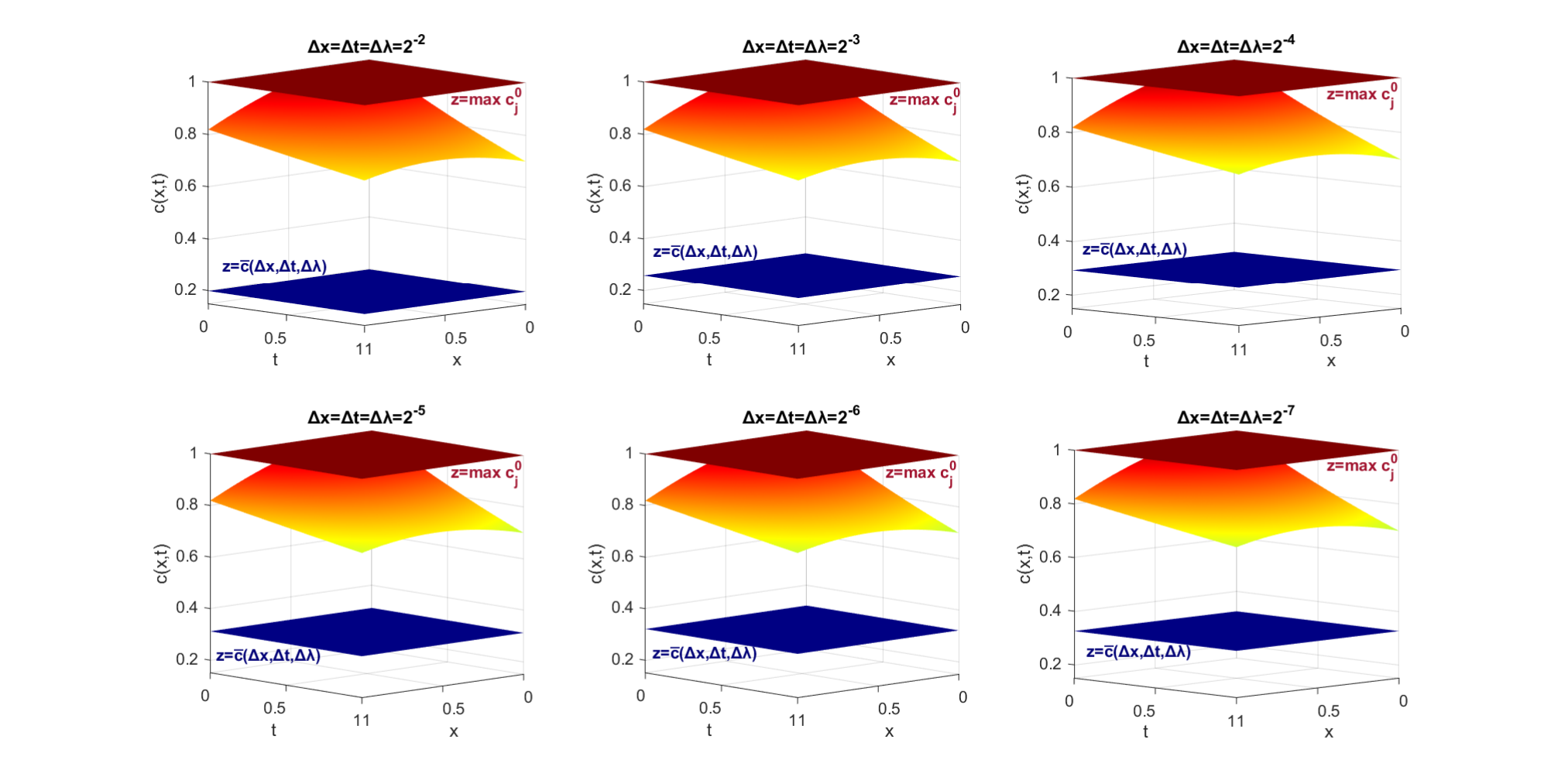}
  \caption{DQ numerical solution of \hyperref[defn: Test_1]{Test 1} with first Gregory rule $(n_0=2)$ and experimental validation of the boundedness property from Theorem \ref{thm:DQ_Properties} for different values of the stepsizes.}
  \label{fig:Thm_Verification_DQ1}
\end{figure}

\medskip

\textbf{Test 2}\label{defn: Test_2} - For our second test, we apply model \eqref{eq:continuous_model} to describe the photoactivation of serotonin ($5$-HT), a biological effector involved in establishing left-right patterning during vertebrate embryonic development \cite{vertebrate1,Lamp+Bio_Sero,Vertebrate2}. The compound BHQ-\!\textit{O}-5HT, derived from $5$-HT through the application of a photoremovable protecting group, undergoes photolysis upon exposure to light, leading to the direct release of the phenol (see Figure \ref{fig:serotonin_structures_combined} for a representation of the molecular structures). This process, extensively exploited in experimental studies, offers high spatial resolution and enables the analysis of processes occurring on timescales shorter than those required for decarboxylation \cite{McLain2015}. 

\begin{figure}[htbp]
 \vspace{-1.8cm}
  \centering
  \begin{minipage}{0.45\textwidth}
    \centering
\includegraphics[width=0.7\linewidth]{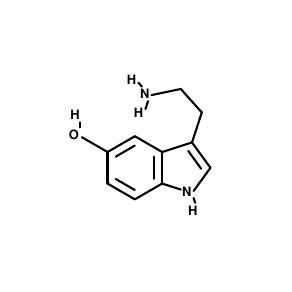}
  \end{minipage}
  \hspace{0.05\textwidth} 
  \begin{minipage}{0.45\textwidth}
    \centering
\includegraphics[width=0.85\linewidth]{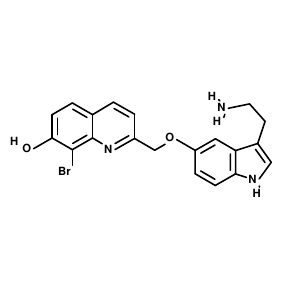}
  \end{minipage}
  \vspace{-4\baselineskip}
  \caption{Chemical structures of serotonin $5$-HT (left) and of photoactivatable serotonin BHQ-\!\textit{O}-5HT (right).}
\label{fig:serotonin_structures_combined}
\end{figure}

Here, we examine the kinetics of an initial concentration \(c^0(x) = 100 \ \mu \text{M}\) of BHQ-\!\textit{O}-5HT, microinjected into a single-cell \textit{Xenopus laevis} embryo with a radius \(L = 5.5 \cdot 10^{-2} \ \text{cm}\). In this setting, $c(x,t)$ represents the concentration of BHQ-\!\textit{O}-5HT at time $t$ and at a distance $z$ from the center of the embryo. Emulating the experimental conditions detailed in \cite{Lamp+Bio_Sero}, we employ the function \(I(\lambda)\), shown in Figure \ref{fig:Test2_I}, to represent the normalized intensity profile of UV radiation emitted by a mercury lamp equipped with two glass filters that restrict the wavelengths to \(365 \pm 15 \ \text{nm}\). Furthermore, we focus on wavelengths between \(\lambda_0 = 300 \ \text{nm}\) and \(\lambda^* = 400 \ \text{nm}\), as higher frequency radiation may potentially damage biological tissues during photoactivation \cite{McLain2015}. The molar absorption coefficients for $5$-HT and BHQ-\!\textit{O}-5HT, extracted from the absorbance data in \cite[Figure 4]{Vecchia_Chimica} and from \cite[Figure S3]{Lamp+Bio_Sero}, respectively, are shown in Figure \ref{fig:Test2_Epsilon}. Additionally, in order to replicate the experimental photolysis course of BHQ-\!\textit{O}-5HT detailed in \cite{McLain2015}, we set \(T = 180\ \text{s}\), \(f(x) = 3 \cdot 10^{-1} x\) and $a_1=a_2=1.$

\begin{figure}[htbp]
  \centering
  \begin{minipage}{0.45\textwidth}
    \centering
    \includegraphics[width=0.9\linewidth]{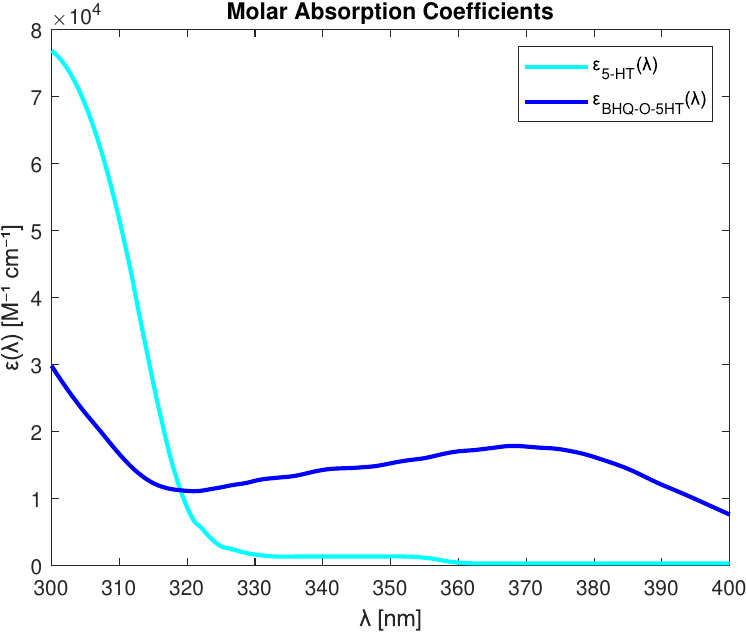}
    \caption{Molar absorption coefficients functions for $5$-HT and BHQ-\!\textit{O}-5HT.}
    \label{fig:Test2_Epsilon}
  \end{minipage}\hspace{0.05\textwidth} 
  \begin{minipage}{0.45\textwidth}
    \centering
    \includegraphics[width=0.9\linewidth]{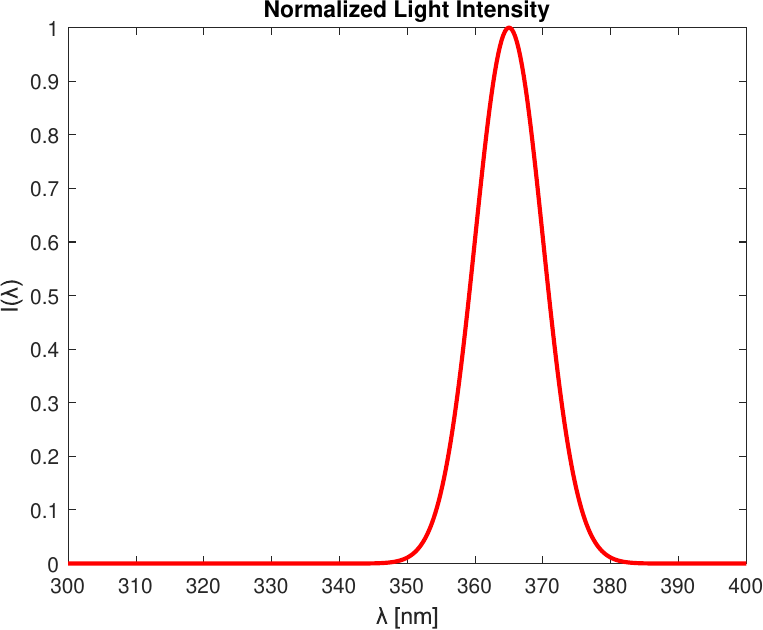}
    \caption{Spectrum of the UV mercury lamp with glass filters.}
    \label{fig:Test2_I}
  \end{minipage}
\end{figure}

The numerical simulation of \hyperref[defn: Test_2]{Test 2} obtained by the three steps, fourth order DQ method \eqref{eq:DQ_method} with second Gregory rule ($n_0=3$) and stepsizes $\Delta x=3.1 \cdot 10^{-4} \ \text{cm},$ $\Delta t=1.0 \ \text{s}$ and $\Delta\lambda=5.5\cdot 10^{-1} \ \text{nm},$ is presented in Figure \ref{fig:Ar2_Reference}. The positivity and monotonicity properties of the numerical solution are evident from the plots. Further insights are provided with Figure \ref{fig:Minimum_Conc}, where the minimum BHQ-\!\textit{O}-5HT concentration over space
\begin{equation*}
    m_c(t_n)=\min_{j=0,\dots,N_x} \!c_j^n\approx\min_{x\in [0,L]}c(x,t_n), \qquad \qquad \qquad n=0,\dots, N_t,
\end{equation*}
is plotted as a positive and non-increasing function of time, for both the PC and DQ schemes. 

\begin{figure}[htbp]
  \centering
  \hspace{-1.6 cm}\begin{minipage}{0.45\textwidth}
    \centering
    \includegraphics[width=1.3\linewidth]{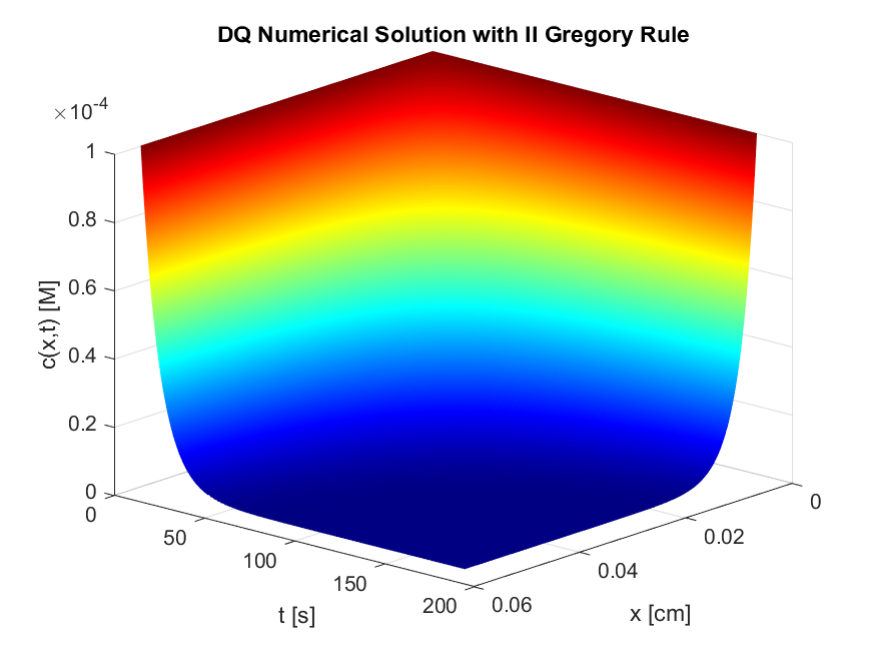}
  \end{minipage}\hspace{0.05\textwidth} 
  \begin{minipage}{0.45\textwidth}
    \centering
    \includegraphics[width=1.26\linewidth]{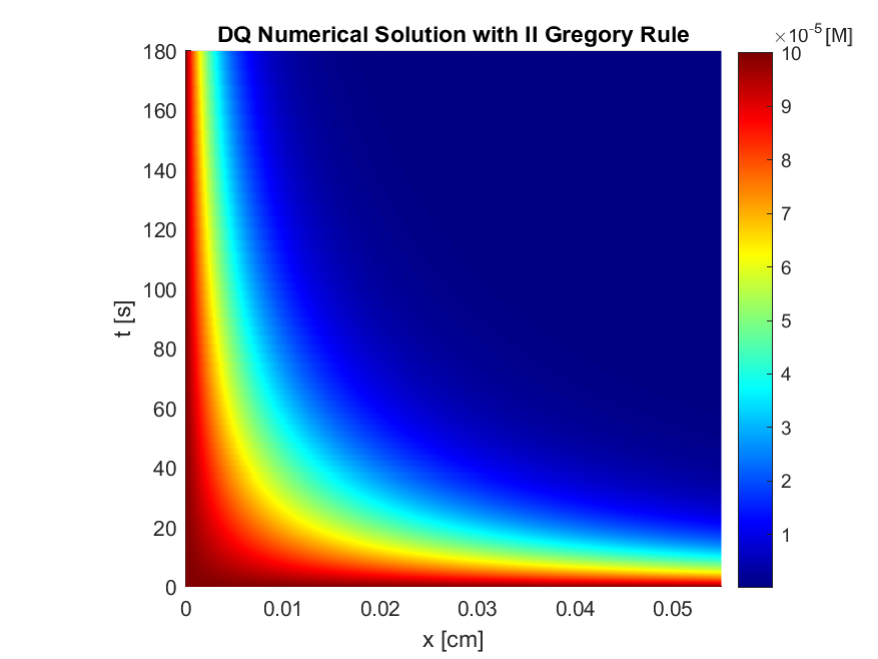}
  \end{minipage}
  \caption{Numerical solution of \hyperref[defn: Test_2]{Test 2} computed by the DQ method \eqref{eq:DQ_method}  with second Gregory rule ($n_0=3$) and stepsizes $\Delta x=3.1 \cdot 10^{-4} \ \text{cm},$ $\Delta t=1.0 \ \text{s}$ and  $\Delta\lambda=5.5\cdot 10^{-1} \ \text{nm}.$}
  \label{fig:Ar2_Reference}
\end{figure}

In order to compare our simulated results with the experimental findings of \cite{McLain2015}, we analyze the percentage reduction in the total concentration of BHQ-\!\textit{O}-5HT at time $t_n$, defined as follows
\begin{equation}
R_c(t_n) = \dfrac{\Delta x}{2}\dfrac{ \left(c_0^n+2\sum_{j=1}^{N_x-1}c_j^n+c_{N_x}^n\right)}{C_0(L)}\approx\frac{\int_{0}^L c(x,t_n) \ dx}{\int_{0}^L c^0(x) \ dx} , \qquad \qquad \qquad n=0,\dots, N_t.
\end{equation}
The PC and DQ simulation outcomes of Figure \ref{fig:Photo_course} correctly reproduce the photolysis dynamics in \cite[Figure 4]{McLain2015}, with the amount of photoactivatable serotonin remaining below $10\%$ after $3$ minutes. 

\begin{figure}[htbp]
  \centering
  \begin{minipage}{0.45\textwidth}
    \centering
    \includegraphics[width=0.85\linewidth]{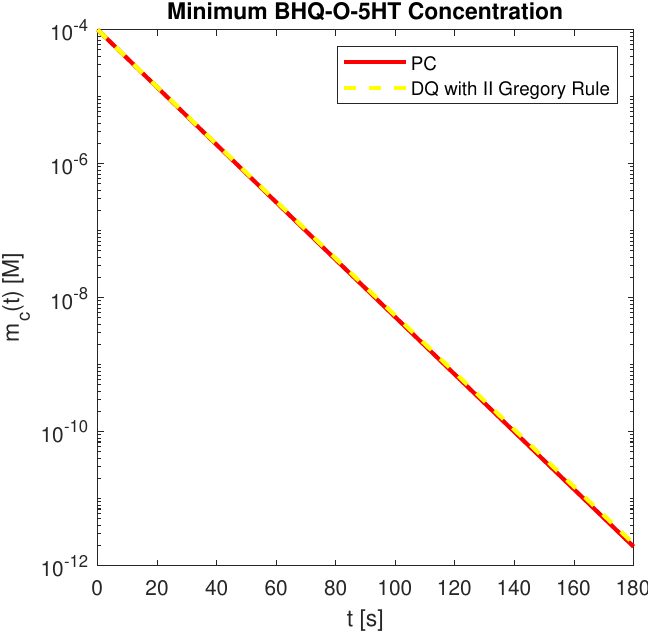}
    \caption{Minimum BHQ-\!\textit{O}-5HT concentration for the PC and DQ schemes applied to \hyperref[defn: Test_2]{Test 2}, plotted on a logarithmic scale for the y-axis.}
    \label{fig:Minimum_Conc}
  \end{minipage}\hspace{0.05\textwidth} 
  \begin{minipage}{0.45\textwidth}
    \centering
    \includegraphics[width=0.85\linewidth]{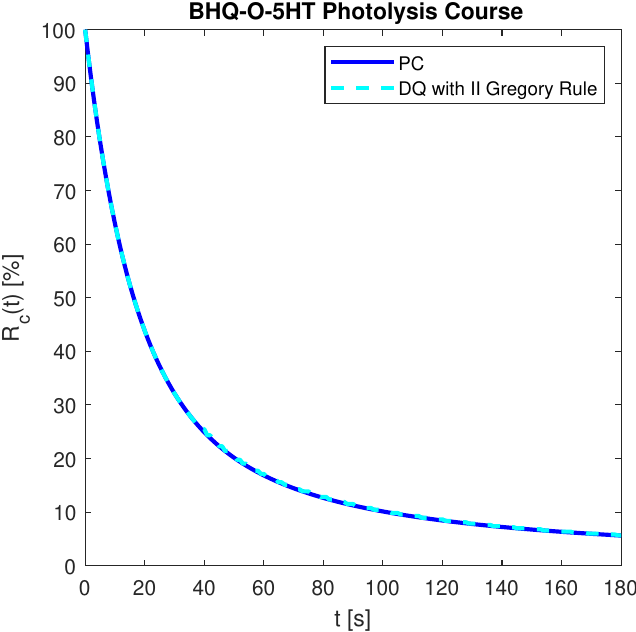}
    \caption{Percentage reduction in the total concentration of BHQ-\!\textit{O}-5HT as a function of time.}
    \label{fig:Photo_course}
  \end{minipage}
\end{figure}

With the aim of emphasizing the importance of the dynamical consistency property in simulations, we extend the analysis of \hyperref[defn: Test_2]{Test 2} to a longer time interval, setting \(T = 1 \ \text{hour}.\) Using stepsizes \(\Delta x = 3.1 \cdot 10^{-4} \ \text{cm},\) \(\Delta t = 5.6 \cdot 10^{-3} \ \text{h}\) and \(\Delta \lambda = 5.5 \cdot 10^{-1} \ \text{nm},\) we compare the positive and monotonicity-preserving NSFD method \eqref{eq:NSFD_scheme} against the following first-order FTRQ scheme
\begin{equation}\label{eq:FTRQ}
    c_j^{n+1} = c_j^n \left(1 - f(x_j) \, \Delta t \, \Delta \lambda \! \sum_{l=0}^{N_\lambda-1} \! \rho \left( \iota\left( \lambda_l, \, C_0(x_j), \, \Delta x \sum_{r=0}^{j-1} c_r^n \right) \right) \right),
\end{equation}
obtained by combining a Forward in Time (FT-) finite difference derivative approximation and a Rectangular Quadrature (-RQ) integral discretization. Figures  \ref{fig:Second_Exp_double} and  \ref{fig:FTRQvsNSFD} provide an overview of the comparison. Although both methods are linearly convergent, the FTRQ scheme fails to reproduce the physical behavior of the phenomenon, leading to oscillations and non-positive concentration values due to the insufficiently small $\Delta t$. In contrast, the NSFD method preserves monotonicity and positivity, demonstrating stability and greater reliability under the same discretization parameters. This property is also characteristic of the PC and DQ methods which, at the cost of increased computational effort, achieve higher orders of accuracy (the corresponding simulations are not included here, for the sake of brevity).

\begin{figure}[htbp]
  \centering
  \hspace{-1.6 cm}\begin{minipage}{0.45\textwidth}
    \centering
    \includegraphics[width=1.3\linewidth]{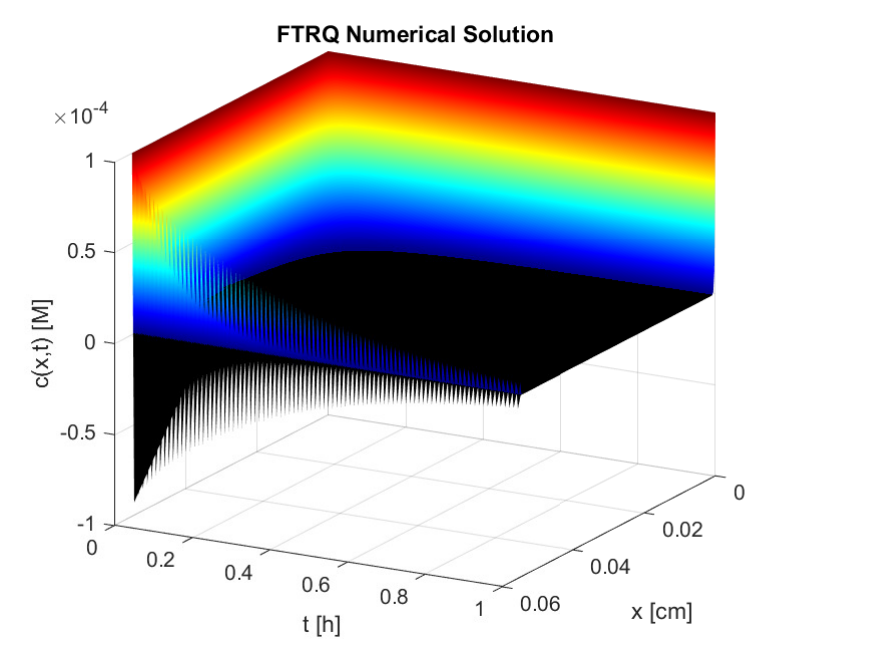}
  \end{minipage}\hspace{0.05\textwidth} 
  \begin{minipage}{0.45\textwidth}
    \centering
    \includegraphics[width=1.3\linewidth]{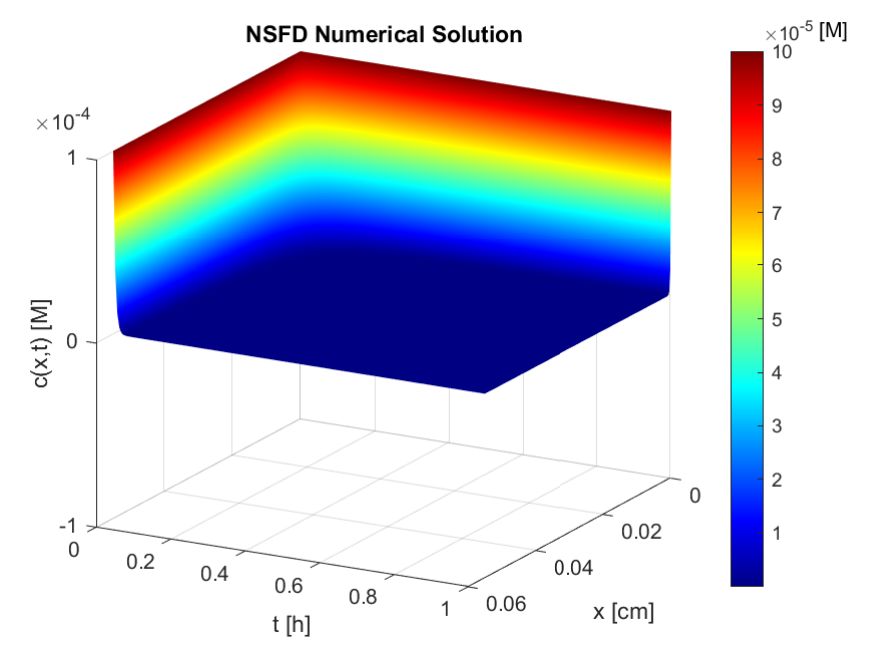}
  \end{minipage}
  \caption{One hour numerical simulation of \hyperref[defn: Test_2]{Test 2} computed using the FTRQ method \eqref{eq:FTRQ} (left) and the NSFD scheme \eqref{eq:NSFD_scheme} (right) with \(\Delta x = 3.1 \cdot 10^{-4} \ \text{cm},\) \(\Delta t = 5.6 \cdot 10^{-3} \ \text{h}\) and \(\Delta \lambda = 5.5 \cdot 10^{-1} \ \text{nm}\). Here, non-positive values of the solutions are plotted in black.}
  \label{fig:Second_Exp_double}
\end{figure}

\begin{figure}[htbp]
\centering
\includegraphics[width=0.9\linewidth]{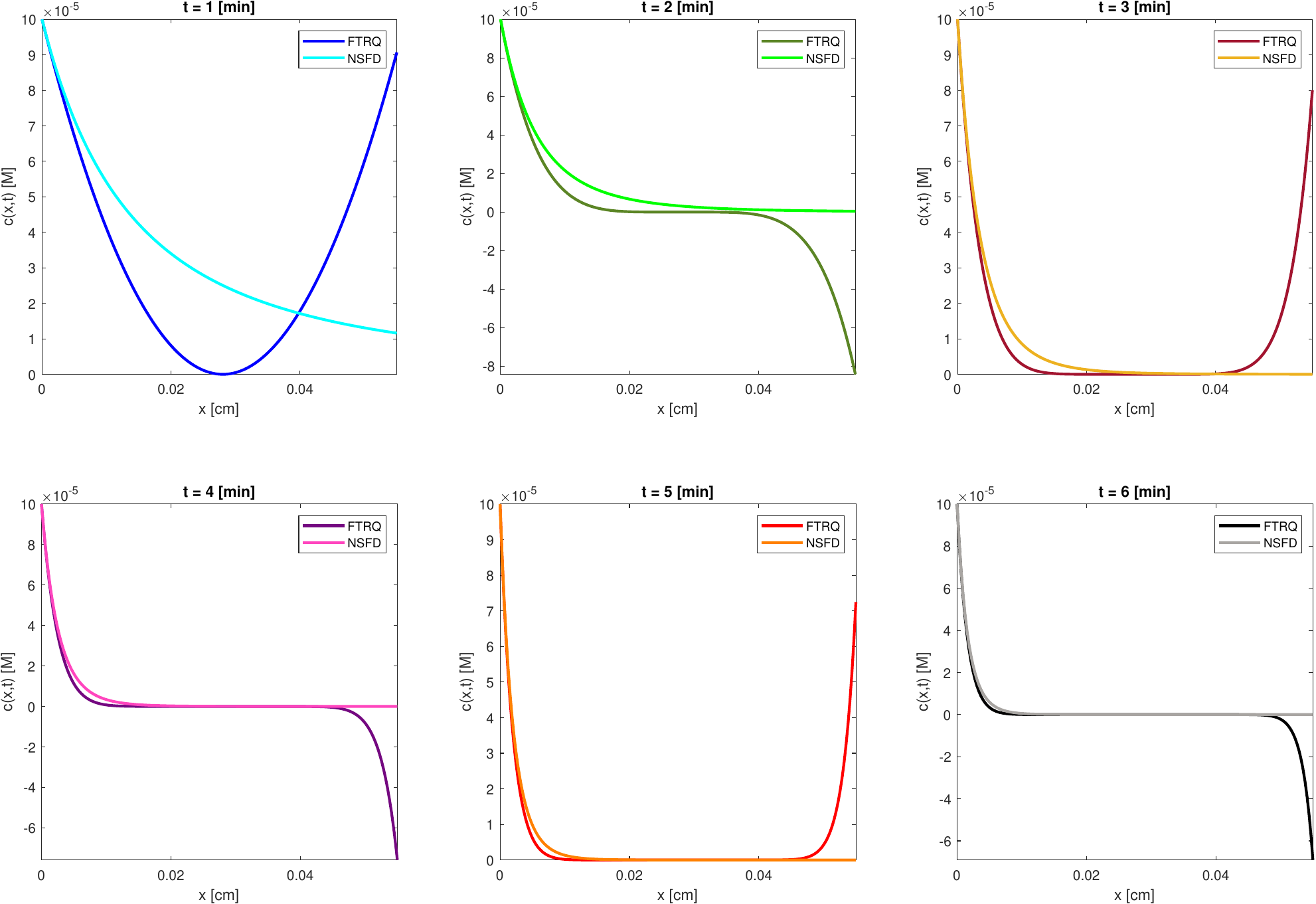}
  \caption{Early stages of NSFD and FTRQ numerical solutions to \hyperref[defn: Test_2]{Test 2} with stepsizes $\Delta x=3.1 \cdot 10^{-4} \ \text{cm},$ $\Delta t=3.3 \cdot 10^{-1} \ \text{min}$ and  $\Delta\lambda=5.5\cdot 10^{-1} \ \text{nm}.$}
  \label{fig:FTRQvsNSFD}
\end{figure}

\medskip

\textbf{Test 3}\label{defn: Test_3} - Our last example addresses the photodegradation of cadmium yellow, a synthetic pigment extensively employed by artists throughout the $19^{\text{th}}$ and $20^{\text{th}}$ centuries \cite{Miro1,Matisse2,Matisse3,Monico_Urlo,Pouyet2015}. Cadmium pigments, primarily composed of cadmium sulfide (CdS), are subject to the following photochemical reaction \cite{monico2018role}
\begin{equation}\label{eq:chemical}
CdS + 4h^+_{surf} + 2H_2O + O_2 \longrightarrow CdSO_4 + 4H^+,
\end{equation}
a process triggered by light and environmental humidity, which leads to the formation of cadmium sulfate (CdSO$_4$) and results in color alteration. Following the arguments in \cite{Ceseri_Natalini_Pezzella} and neglecting the influence of secondary reactants and products, the kinetics of CdS can be effectively modeled through the system \eqref{eq:continuous_model}. Specifically, we simulate the degradation mechanism of an initial concentration of $c^0(x)=3.34 \cdot 10^{-2} \ \text{mol} \ \text{cm}^{-3}$ on a painted layer of depth $L=7.00 \cdot 10^{-3} \, \text{cm}.$ Since cadmium sulfide behaves as a semiconductor with a defined band-gap energy of $E_{bg}=2.42$ eV \cite{Band_Gap}, we set the upper wavelength bound to $\lambda_*=512.33 \ \text{nm}$, which corresponds to supra-band-gap light radiation. The molar absorptivity \(\varepsilon_{\text{CdS}}(\lambda)\) of cadmium sulfide is computed using the relation \cite{gobrecht2015combining}  
\begin{equation*}
    \varepsilon_{\text{CdS}}(\lambda) \, c^{0} d_L = -\log\left(R(\lambda)\right), \qquad \qquad \quad \lambda \in [\lambda_0, \lambda_*],
\end{equation*}
where \(R(\lambda)\) corresponds to the hex-CdS diffuse reflectance UV-Vis spectral data provided in \cite[Supporting Information, Sec. 2.1]{monico2018role}. The parameter \(d_L = 5\cdot 10^{-4} \, \text{cm}\) indicates the thickness of the cadmium sulfate top crust, beyond which light penetration is significantly reduced due to CdSO\(_4\) formation. As for the cadmium sulfate, in absence of experimental data, we assume $\varepsilon_{\text{CdSO}_4}(\lambda)=4\varepsilon_{\text{CdS}}(\lambda)$ (we refer to Figure \ref{fig:Test3_Epsilon} for the plots). Furthermore, we determine the function \(I(\lambda)\), displayed in Figure \ref{fig:Test3_I}, by applying a local quadratic regression to the xenon lamp emission data of \cite[Fig. 1]{Monico_Dati_per_I}.

\begin{figure}[htbp]
  \centering    
  \begin{minipage}{0.45\textwidth}
    \centering
    \includegraphics[width=0.9\linewidth]{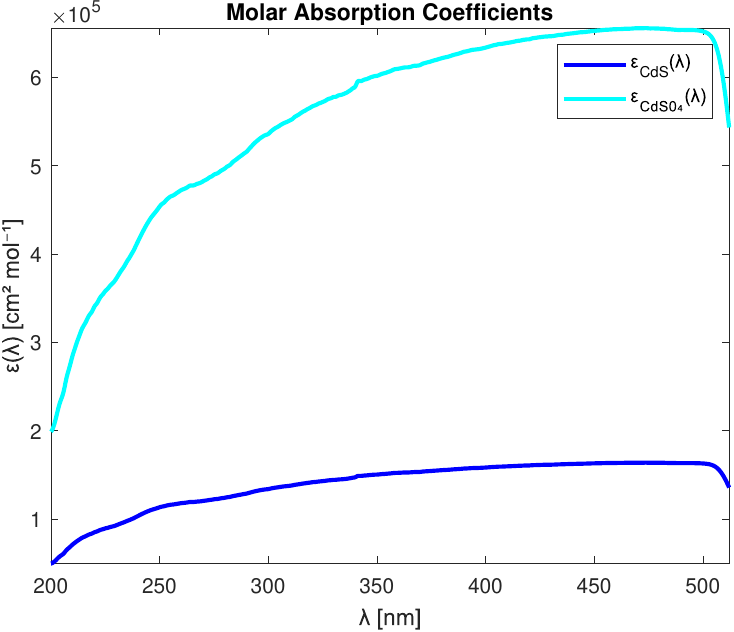}
    \caption{Molar absorption coefficients functions for CdS and CdSO$_4$.}
    \label{fig:Test3_Epsilon}
  \end{minipage}\hspace{0.05\textwidth} 
  \begin{minipage}{0.45\textwidth}
    \centering
    \includegraphics[width=0.9\linewidth]{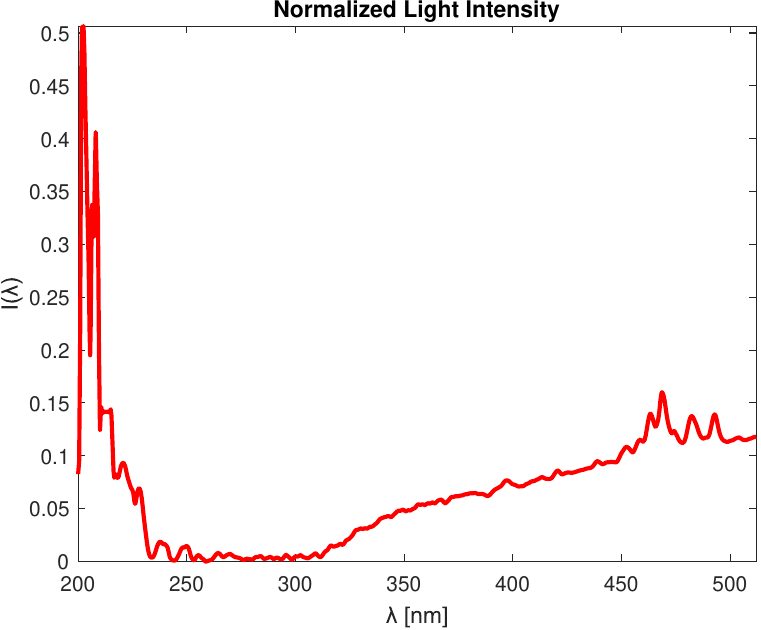}
    \caption{Normalized spectrum of the UV xenon lamp.}
    \label{fig:Test3_I}
  \end{minipage}
\end{figure}

The environmental humidity required to trigger the reaction is modeled by the following water profile
\begin{equation*}
    f(x)=10^{-3}\left(1-\dfrac{w_0}{w_{*}}\right)(L-x), \qquad \qquad \qquad x\in [0,L],
\end{equation*}
which decreases linearly with depth. In this context, $w_0=5.77 \cdot 10^{-7} \ \text{mol} \ \text{cm}^{-3}$ denotes the lowest level of humidity below which the reaction does not occur and $w_*=1.22 \cdot 10^{-6} \ \text{mol} \ \text{cm}^{-3}$ is a given reference value. Figure \ref{fig:Ar2_Reference} illustrates the simulation results of \hyperref[defn: Test_3]{Test 3}, obtained by using the fourth order DQ method \eqref{eq:DQ_method} with the second Gregory rule ($n_0=3$) and stepsizes $\Delta x=1.4 \cdot 10^{-4} \ \text{cm}$, $\Delta t=1.0 \ \text{s}$ and $\Delta\lambda=1 \ \text{nm}$. In compliance with experimental observations (cf. \cite{monico2018role,Monico_Dati_per_I}), the degradation phenomenon is predominantly confined to a shallow region of the painting (i.e., for $x\leq d_L$) due to the formation of a CdSO$_4$ reflective layer which significantly attenuates light penetration. 

\begin{figure}[htbp]
  \centering
  \hspace{-1.65 cm}\begin{minipage}{0.45\textwidth}
    \centering
    \includegraphics[width=1.3\linewidth]{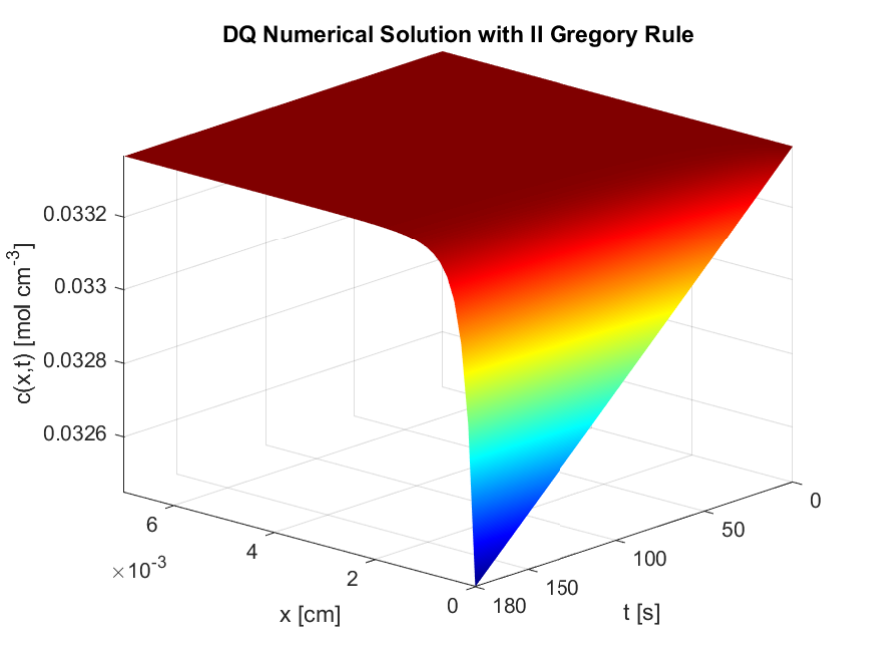}
  \end{minipage}\hspace{0.09\textwidth} 
  \begin{minipage}{0.45\textwidth}
    \centering
    \includegraphics[width=1.26\linewidth]{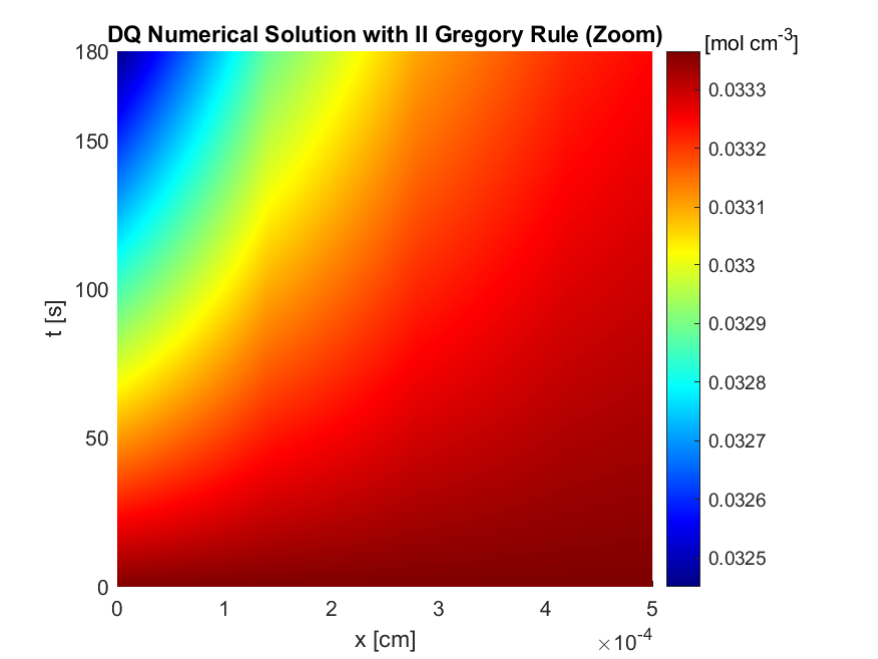}
  \end{minipage}
  \caption{Numerical solution of \hyperref[defn: Test_3]{Test 3} computed by the DQ method \eqref{eq:DQ_method}  with second Gregory rule ($n_0=3$) and stepsizes $\Delta x=1.4 \cdot 10^{-4} \ \text{cm},$ $\Delta t=1.0 \ \text{s}$ and  $\Delta\lambda=1 \ \text{nm}.$}
  \label{fig:Ar3_Reference}
\end{figure}

An equivalent dimensionless reformulation of \hyperref[defn: Test_3]{Test 3} is simulated in \cite{Ceseri_Natalini_Pezzella} using a dynamically consistent PCRQ scheme, which employs the RQ integrator \eqref{eq:RQ} for the predictor phase and the trapezoidal rule as a corrector. Being quadratically convergent, the PCRQ method is undoubtedly  less accurate than the higher-order DQ scheme \eqref{eq:DQ_method}. Here, we compare the PCRQ discretization with the second-order PC method \eqref{eq:PC}, which shares the same convergence properties and similar computational demands. Motivated by the results of \hyperref[defn: Test_1]{Test 1}, we select $\phi_2(\Delta t)$ in \eqref{eq:Denominator_Functions} as the denominator function for the PC scheme \eqref{eq:PC}. Figure \ref{fig:Ar3_Confronto} shows the mean-space error as a function of time for both the PCRQ and PC simulations with $\Delta x=1.4 \cdot 10^{-4} \ \text{cm},$ $\Delta t=0.5 \ \text{s}$ and $\Delta\lambda=1 \ \text{nm}$. Specifically, the error is defined as
\begin{equation} 
    e(t_n)=\dfrac{1}{N_x}\sum_{j=0}^{N_z} \left|c_j^n-\bar{\mathcal{C}}_j^n\right|, \qquad \qquad \qquad n=0,\dots, N_t,
\end{equation}
where $\{\bar{\mathcal{C}}_j^n\}$ denotes the reference solution computed using the fourth-order DQ method \eqref{eq:DQ_method} with $n_0=3,$ second Gregory weights and the same stepsizes values. From the plots in Figure \ref{fig:Ar3_Confronto}, it is evident that the PC integrator \eqref{eq:PC} outperforms the PCRQ one in \cite{Ceseri_Natalini_Pezzella}, as it yields a lower mean-space error at each time step, for the same discretization steplengths.

\begin{figure}[htbp]
\centering
\includegraphics[width=0.55\linewidth]{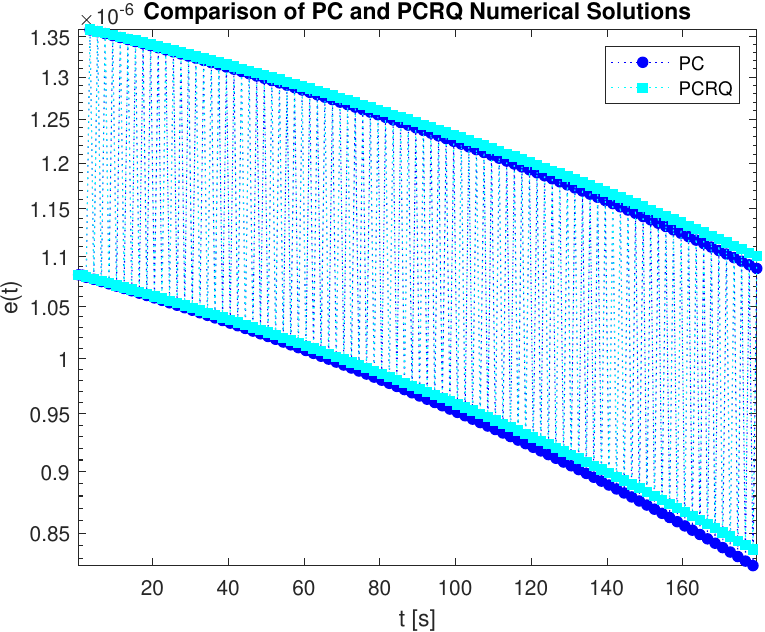}
  \caption{Time evolution of the mean-space error for the PC and PCRQ simulations of \hyperref[defn: Test_3]{Test 3}, computed with $\Delta x=1.4 \cdot 10^{-4} \ \text{cm},$ $\Delta t=0.5 \ \text{s}$ and $\Delta\lambda=1 \ \text{nm}$.}
  \label{fig:Ar3_Confronto}
\end{figure}

\section{Conclusions}\label{sec:Conclusions}
In this work, we designed three classes of dynamically consistent numerical methods for the integro-differential photochemical model \eqref{eq:continuous_model}. A key property of the proposed discretizations is their capacity to preserve positivity in the numerical solutions, regardless of the chosen stepsizes values. The different methods provide a flexible simulation framework to meet various requirements, depending on the phenomenon under investigation and the specific application. The first order non-standard finite difference (NSFD) scheme \eqref{eq:NSFD_scheme} and the quadratically convergent Predictor-Corrector (PC) method \eqref{eq:PC} are straightforward to implement and ensure the monotonicity of the numerical solution, as well. The Direct Quadrature (DQ) integrator \eqref{eq:DQ_method}, on the other hand, achieves higher accuracy and would generally be preferred for advanced simulations. However, due to its inherent nonlinearity and greater computational complexity, the DQ scheme is less suited for model calibration, which typically involves numerous repeated simulations. In such cases, the NSFD and PC methods, although less accurate, offer a more efficient and feasible alternative.

Future research could explore additional realistic applications of the non-local model \eqref{eq:continuous_model}, which already demonstrates significant versatility. From a numerical standpoint, the proposed methods could undergo further refinement to improve computational efficiency. A theoretical investigation aimed at defining \textit{ad hoc} denominator functions for the NSFD framework, specifically tailored to the non-local nature of the problem, could enhance accuracy for larger time steps. Furthermore, the embedding of DQ schemes with $n_0>1$ into the PC approach may yield high-order explicit methods that combine accuracy with lower demands. While these extensions appear ideally promising, their practical implementation and computational efficiency require careful assessment, which we plan to address in future work.

\bibliography{Preprint_Pezzella}

\end{document}